\pdfoutput=1
\documentclass[11pt, a4paper, twoside,reqno]{amsart}
\usepackage[centering, totalwidth = 390pt, totalheight = 610pt]{geometry}
\usepackage{amssymb, amsmath, amsthm}
\usepackage{microtype, stmaryrd, url, lmodern, eucal, extdash}
\usepackage[shortlabels]{enumitem}
\usepackage[latin1]{inputenc}
\usepackage{color}
\definecolor{darkgreen}{rgb}{0,0.45,0}
\usepackage[colorlinks,citecolor=darkgreen,linkcolor=darkgreen]{hyperref}
\usepackage[capitalize]{cleveref}
\usepackage[british]{babel}

\makeatletter
\def\@cite#1#2{[{#1\if@tempswa ,~#2\fi}]}
\makeatother
\allowdisplaybreaks

\DeclareMathAlphabet{\mathbf}{OT1}{cmr}{b}{n}
\DeclareFontSeriesDefault[rm]{bf}{b}

\usepackage[arrow, matrix, tips, curve, graph, rotate]{xy}
\SelectTips{cm}{10}

\makeatletter
\def\matrixobject@{%
  \edef \next@{={\DirectionfromtheDirection@ }}%
  \expandafter \toks@ \next@ \plainxy@
  \let\xy@@ix@=\xyq@@toksix@
  \xyFN@ \OBJECT@}
\let\xy@entry@@norm=\entry@@norm
\def\entry@@norm@patched{%
  \let\object@=\matrixobject@
  \xy@entry@@norm }
\AtBeginDocument{\let\entry@@norm\entry@@norm@patched}
\makeatother

\newcommand{\twocong}[2][0.5]{\ar@{}[#2] \save ?(#1)*{\cong}\restore}
\newcommand{\twoeq}[2][0.5]{\ar@{}[#2] \save ?(#1)*{=}\restore}
\newcommand{\rtwocell}[3][0.5]{\ar@{}[#2] \ar@{=>}?(#1)+/l 0.2cm/;?(#1)+/r 0.2cm/^{#3}}
\newcommand{\ltwocell}[3][0.5]{\ar@{}[#2] \ar@{=>}?(#1)+/r 0.2cm/;?(#1)+/l 0.2cm/^{#3}}
\newcommand{\ltwocello}[3][0.5]{\ar@{}[#2] \ar@{=>}?(#1)+/r 0.2cm/;?(#1)+/l 0.2cm/_{#3}}
\newcommand{\dtwocell}[3][0.5]{\ar@{}[#2] \ar@{=>}?(#1)+/u  0.2cm/;?(#1)+/d 0.2cm/^{#3}}
\newcommand{\dltwocell}[3][0.5]{\ar@{}[#2] \ar@{=>}?(#1)+/ur  0.2cm/;?(#1)+/dl 0.2cm/^{#3}}
\newcommand{\drtwocell}[3][0.5]{\ar@{}[#2] \ar@{=>}?(#1)+/ul  0.2cm/;?(#1)+/dr 0.2cm/^{#3}}
\newcommand{\dthreecell}[3][0.5]{\ar@{}[#2] \ar@3{->}?(#1)+/u  0.2cm/;?(#1)+/d 0.2cm/^{#3}}
\newcommand{\utwocell}[3][0.5]{\ar@{}[#2] \ar@{=>}?(#1)+/d 0.2cm/;?(#1)+/u 0.2cm/_{#3}}
\newcommand{\dtwocelltarg}[3][0.5]{\ar@{}#2 \ar@{=>}?(#1)+/u  0.2cm/;?(#1)+/d 0.2cm/^{#3}}
\newcommand{\utwocelltarg}[3][0.5]{\ar@{}#2 \ar@{=>}?(#1)+/d  0.2cm/;?(#1)+/u 0.2cm/_{#3}}

\newdir{(}{{}*!<0em,-.14em>-\cir<.14em>{l^r}}
\newdir{ (}{{}*!/-5pt/\dir{(}}
\newdir{ >}{{}*!/-5pt/\dir{>}}

\DeclareMathOperator{\im}{im}

\newcommand{\cat}[1]{\mathrm{\mathcal #1}}
\newcommand{\thg}{{\mathord{\text{--}}}}
\newcommand{\quot}{\delimiter"502F30E\mathopen{}}

\newcommand{\dbr}[1]{\mathord{\left\llbracket{#1}\right\rrbracket}}
\newcommand{\res}[2]{\left.{#1}\right|_{#2}}

\newcommand{\spn}[1]{{\langle{#1}\rangle}}

\newcommand{\cd}[2][]{\vcenter{\hbox{\xymatrix#1{#2}}}}
\newcommand{\mdl}[1][A]{{\boldsymbol{#1}}}

\renewcommand{\phi}{\varphi}
\newcommand{\A}{{\mathcal A}}
\newcommand{\B}{{\mathcal B}}
\newcommand{\C}{{\mathcal C}}

\renewcommand{\H}{{\mathcal H}}

\newcommand{\J}{{\mathcal J}}
\newcommand{\K}{{\mathcal K}}

\newcommand{\M}{{\mathcal M}}

\renewcommand{\P}{{\mathcal P}}

\let\sec=\S
\renewcommand{\S}{{\mathcal S}}
\newcommand{\T}{{\mathcal T}}
\newcommand{\U}{{\mathcal U}}
\newcommand{\V}{{\mathcal V}}

\newcommand{\xtor}[1]{\cdl[@1]{{} \ar[r]|-{\object@{|}}^{#1} & {}}}

\makeatletter

\def\hookleftarrowfill@{\arrowfill@\leftarrow\relbar{\relbar\joinrel\rhook}}
\def\twoheadleftarrowfill@{\arrowfill@\twoheadleftarrow\relbar\relbar}
\def\leftbararrowfill@{\arrowdoublefill@{\leftarrow\mkern-5mu}\relbar\mapstochar\relbar\relbar}
\def\Leftbararrowfill@{\arrowdoublefill@{\Leftarrow\mkern-2mu}\Relbar\Mapstochar\Relbar\Relbar}
\def\leftringarrowfill@{\arrowdoublefill@{\leftarrow\mkern-3mu}\relbar{\mkern-3mu\circ\mkern-2mu}\relbar\relbar}
\def\lefttriarrowfill@{\arrowfill@{\mathrel\triangleleft\mkern0.5mu\joinrel\relbar}\relbar\relbar}
\def\Lefttriarrowfill@{\arrowfill@{\mathrel\triangleleft\mkern1mu\joinrel\Relbar}\Relbar\Relbar}

\def\hookrightarrowfill@{\arrowfill@{\lhook\joinrel\relbar}\relbar\rightarrow}
\def\twoheadrightarrowfill@{\arrowfill@\relbar\relbar\twoheadrightarrow}
\def\rightbararrowfill@{\arrowdoublefill@{\relbar\mkern-0.5mu}\relbar\mapstochar\relbar\rightarrow}
\def\Rightbararrowfill@{\arrowdoublefill@{\Relbar\mkern-2mu}\Relbar\Mapstochar\Relbar\Rightarrow}
\def\rightringarrowfill@{\arrowdoublefill@\relbar\relbar{\mkern-2mu\circ\mkern-3mu}\relbar{\mkern-3mu\rightarrow}}
\def\righttriarrowfill@{\arrowfill@\relbar\relbar{\relbar\joinrel\mkern0.5mu\mathrel\triangleright}}
\def\Righttriarrowfill@{\arrowfill@\Relbar\Relbar{\Relbar\joinrel\mkern1mu\mathrel\triangleright}}

\def\leftrightarrowfill@{\arrowfill@\leftarrow\relbar\rightarrow}
\def\mapstofill@{\arrowfill@{\mapstochar\relbar}\relbar\rightarrow}

\renewcommand*\xleftarrow[2][]{\ext@arrow 20{20}0\leftarrowfill@{#1}{#2}}
\providecommand*\xLeftarrow[2][]{\ext@arrow 60{22}0{\Leftarrowfill@}{#1}{#2}}
\providecommand*\xhookleftarrow[2][]{\ext@arrow 10{20}0\hookleftarrowfill@{#1}{#2}}
\providecommand*\xtwoheadleftarrow[2][]{\ext@arrow 60{20}0\twoheadleftarrowfill@{#1}{#2}}
\providecommand*\xleftbararrow[2][]{\ext@arrow 10{22}0\leftbararrowfill@{#1}{#2}}
\providecommand*\xLeftbararrow[2][]{\ext@arrow 50{24}0\Leftbararrowfill@{#1}{#2}}
\providecommand*\xleftringarrow[2][]{\ext@arrow 10{26}0\leftringarrowfill@{#1}{#2}}
\providecommand*\xlefttriarrow[2][]{\ext@arrow 80{24}0\lefttriarrowfill@{#1}{#2}}
\providecommand*\xLefttriarrow[2][]{\ext@arrow 80{24}0\Lefttriarrowfill@{#1}{#2}}

\renewcommand*\xrightarrow[2][]{\ext@arrow 01{20}0\rightarrowfill@{#1}{#2}}
\providecommand*\xRightarrow[2][]{\ext@arrow 04{22}0{\Rightarrowfill@}{#1}{#2}}
\providecommand*\xhookrightarrow[2][]{\ext@arrow 00{20}0\hookrightarrowfill@{#1}{#2}}
\providecommand*\xtwoheadrightarrow[2][]{\ext@arrow 03{20}0\twoheadrightarrowfill@{#1}{#2}}
\providecommand*\xrightbararrow[2][]{\ext@arrow 01{22}0\rightbararrowfill@{#1}{#2}}
\providecommand*\xRightbararrow[2][]{\ext@arrow 04{24}0\Rightbararrowfill@{#1}{#2}}
\providecommand*\xrightringarrow[2][]{\ext@arrow 01{26}0\rightringarrowfill@{#1}{#2}}
\providecommand*\xrighttriarrow[2][]{\ext@arrow 07{24}0\righttriarrowfill@{#1}{#2}}
\providecommand*\xRighttriarrow[2][]{\ext@arrow 07{24}0\Righttriarrowfill@{#1}{#2}}

\providecommand*\xmapsto[2][]{\ext@arrow 01{20}0\mapstofill@{#1}{#2}}
\providecommand*\xleftrightarrow[2][]{\ext@arrow 10{22}0\leftrightarrowfill@{#1}{#2}}
\providecommand*\xLeftrightarrow[2][]{\ext@arrow 10{27}0{\Leftrightarrowfill@}{#1}{#2}}

\makeatother

\crefname{equation}{}{}
\crefname{Lemma}{Lemma}{Lemmas}
\crefname{Thm}{Theorem}{Theorems}
\crefname{Defn}{Definition}{Definitions}
\crefname{Not}{Notation}{Notations}
\crefname{Ex}{Example}{Examples}
\crefname{Exs}{Examples}{Examples}
\crefname{sec}{Section}{Sections}
\crefname{Prop}{Proposition}{Propositions}
\crefname{Rk}{Remark}{Remarks}

\numberwithin{equation}{section}

\theoremstyle{plain}
\newtheorem{Thm}{Theorem}[section]
\newtheorem{Prop}[Thm]{Proposition}

\newtheorem{Lemma}[Thm]{Lemma}
\newtheorem*{Thm*}{Theorem}

\theoremstyle{definition}
\newtheorem{Defn}[Thm]{Definition}
\newtheorem{Not}[Thm]{Notation}

\newtheorem{Rk}[Thm]{Remark}

\newcommand{\AT}{\cat{\T hy}}
\newcommand{\HA}{\cat{\H\A ff}}
\newcommand{\GBA}{\mathrm{gr}\cat{\B \A lg}}
\newcommand{\BA}{\cat{B \A lg}}
\newcommand{\Mon}{\cat{Mon}}
\newcommand{\Un}{\cat{\U n}}
\newcommand{\HAUn}{\cat{\H\A ff\text{--}\U n}}
\newcommand{\HAUnf}{\cat{\H\A ff\text{--}\U n^\omega}}
\newcommand{\Var}{\cat{Var}}
\newcommand{\CCVar}{\mathrm{cc}\cat{\V ar}}
\newcommand{\Top}{\J}
\newcommand{\BJ}{B_{\J}}

\newcommand{\dc}[2]{\left[\smash{{#1} \mathbin{\mid}{#2} }\right]}
\newcommand{\BM}{{\dc B M}}
\newcommand{\BJM}{{\dc {\BJ} M}}

\newcommand{\bc}{\mathbin{\bowtie}}
\newcommand{\sheq}[2]{{\dbr{{#1}\mathrel{\!\texttt{\upshape =}\!}{#2}}}}

\begin{document}
\leftmargini=2em \title[Cartesian closed varieties I]{Cartesian closed
  varieties I:\\The classification theorem} \author{Richard Garner} \address{School of
  Math.~\& Phys.~Sciences, Macquarie University, NSW 2109,
  Australia} \email{richard.garner@mq.edu.au}

\date{\today}

\thanks{The support of Australian Research Council grant DP190102432
  is gratefully acknowledged.}

\begin{abstract}
  In 1990, Johnstone gave a syntactic characterisation of the
  equational theories whose associated varieties are cartesian closed.
  Among such theories are all \emph{unary} theories---whose models are
  sets equipped with an action by a monoid $M$---and all
  \emph{hyperaffine} theories---whose models are sets with an action
  by a Boolean algebra $B$. We improve on Johnstone's result by
  showing that an equational theory is cartesian closed just when its
  operations have a unique hyperaffine--unary decomposition. It
  follows that any non-degenerate cartesian closed variety is a
  variety of sets equipped with compatible actions by a monoid $M$ and
  a Boolean algebra $B$; this is the classification theorem of the title.
\end{abstract}
\maketitle
\setcounter{tocdepth}{1}
\tableofcontents
\section{Introduction}

In~\cite{Johnstone1990Collapsed}, Johnstone considered the following
very natural question: when is a variety---by which we mean the
category of models of a single-sorted equational algebraic theory---a
cartesian closed category? This was, in fact, a follow-up to an
earlier question---``when is a variety a topos?''---asked by Johnstone
in~\cite{Johnstone1985When}, with the answers in the two cases turning
out to be surprisingly similar. The solutions Johnstone provides are
\emph{syntactic recognition theorems}, giving necessary and sufficient
conditions on the operations of an equational theory for the variety
it presents to be cartesian closed or a topos. We recall the cartesian
closed result as Theorem~\ref{thm:2} below, and the reader will
readily observe that, while a little delicate, the conditions involved
are straightforward enough to be practically useful; and indeed, a
very similar set of conditions finds computational application
in~\cite{Pattinson2015Sound}.

Be this as it may, Johnstone's conditions do little to help us
delineate the scope of the cartesian closed varieties. Much as we can
say that every Grothendieck topos is the topos of sheaves on a site,
we would like to say that every cartesian closed variety is \dots\ and
filling this gap would amount to providing a \emph{semantic
  classification theorem} for cartesian closed varieties. This is one
of the main objectives of this paper: we will show that every
cartesian closed variety is the variety of sets endowed with two
actions, one by a monoid $M$ and one by a Boolean algebra $B$, which
interact in a suitable way. Thus, our classification shows that any
cartesian closed variety is a kind of ``bicrossed product'' of the
variety of $M$-sets, which as a presheaf category, is well known to be
cartesian closed; and the variety of \emph{$B$-sets}, as introduced
in~\cite{Bergman1991Actions} and recalled in
Section~\ref{sec:b-sets-finitary}, which was shown to be cartesian
closed in~\cite[Example~8.8]{Johnstone1990Collapsed}.

Our semantic classification theorem will be obtained by way of a
\emph{syntactic classification theorem} derived from Johnstone's
recognition theorem. To motivate this result, observe first that the
cartesian closed varieties of $M$-sets are precisely those which can
be presented by \emph{unary} algebraic theories, that is, theories
whose operations and equations are all of arity $1$. On the other
hand, as shown in~\cite{Johnstone1990Collapsed} and recalled in
Section~\ref{sec:unary-algebr-theor}, the cartesian closed varieties
of $B$-sets are precisely those which are presented by
\emph{hyperaffine} algebraic theories, that is ones whose every
operation is hyperaffine (Definition~\ref{def:11}); here, for (say) a
ternary operation $f$, hyperaffineness asserts that $f(x,x,x) = x$,
i.e., $f$ is affine, and moreover that:
\begin{equation*}
  f(f(x_{11}, x_{12}, x_{13}), f(x_{21}, x_{22}, x_{23}), f(x_{31}, x_{32}, x_{33})) = f(x_{11}, x_{22}, x_{33})\rlap{ .}
\end{equation*}
Our syntactic classification theorem (Theorem~\ref{thm:3}) now states
that:
\begin{Thm*}
  An equational theory presents a cartesian closed variety if, and
  only if, every operation $f$ has a unique decomposition as a
  hyperaffine operation $h$ applied to a unary one $m$, i.e.,
  $f(x_1, x_2,\dots, x_n) = h(m(x_1), m(x_2),\dots, m(x_n))$.
\end{Thm*}
The proof of this result is simply Johnstone's recognition theorem
together with a little calculation, but we should note that our result
does not supplant Johnstone's, but rather complements it: for while
our condition may be simpler to state, it is harder to check if one
wants to determine if a given variety is cartesian closed.

Where our formulation comes into its own is in deriving our semantic
classification theorem. If $\mathbb{T}$ presents a cartesian closed
variety, then the syntactic classification theorem tell us that its
operations are completely determined by the monoid of unary operations
$M$ together with the subtheory $\mathbb{H} \subseteq \mathbb{T}$ of
hyperaffine operations. However, just knowing these is not enough to
recover the \emph{substitution} of our equational theory, and so we
must also record the manner in which $\mathbb{H}$ and $M$ act on each
other by substitution. This leads to what we term a \emph{matched pair
  of theories} $\dc {\mathbb{H}} M$ (Definition~\ref{def:14}), and our
second main result (Theorem~\ref{thm:4}):
\begin{Thm*}
  The category of non-degenerate cartesian closed varieties is
  equivalent to the category of non-degenerate matched pairs of
  theories.
\end{Thm*}
Applying the correspondence between hyperaffine algebraic theories
$\mathbb{H}$ and theories of $B$-sets over Boolean algebras $B$ now
transforms each matched pair of theories into what we term a
\emph{matched pair of algebras} $\BM$ (Definition~\ref{def:16}). This
involves a Boolean algebra $B$ and a monoid $M$ such that $M$ is a
$B$-set, $B$ is an $M$-set, and various further equational axioms
hold. In fact, this structure has been studied in the literature: in
the nomenclature of~\cite{Jackson2009Semigroups}, we would say that
\emph{$M$ is a $B$-monoid}. Concomitantly, we have a notion of
\emph{$\BM$-set} (Definition~\ref{def:17}), which is a set equipped
with $B$-action and $M$-action in a manner which is compatible with
the $B$-action on $M$ and the $M$-action on $B$. In terms of this, we
finally obtain our semantic classification result
(Theorem~\ref{thm:5}):
\begin{Thm*}
  The category of non-degenerate cartesian closed varieties is
  equivalent to the category of non-degenerate matched pairs of
  algebras via an equivalence which identifies the matched pair $\BM$
  with the cartesian closed variety of $\BM$-sets.
\end{Thm*}
There is one point we should clarify about the preceding result. As
stated, it is only valid for varieties and equational theories which
are \emph{finitary}, i.e., generated by operations of finite arity.
However, in the paper proper, it will also be valid in the
\emph{infinitary} case; and the adjustments needed to account for this
are entirely confined to the Boolean algebra side of things. Indeed,
whereas \emph{finitary} hyperaffine theories correspond to Boolean
algebras $B$, arbitary hyperaffine theories correspond to
\emph{strongly zero-dimensional locales}; these are locales (=
complete Heyting algebras) in which every cover can be refined to a
partition. Two different proofs of this correspondence can be found
in~\cite[Example~8.8]{Johnstone1990Collapsed} and in~\cite[\sec
2]{Johnstone1997Cartesian}; we in fact provide a third proof
(Theorem~\ref{thm:1}), but with respect to a slightly different
presentation of strongly zero-dimensional locales, inspired
by~\cite{Name2013Boolean}: we consider Boolean algebras $B$ equipped
with a collection $\J$ of well-behaved partitions of $B$, under axioms
which make them correspond to strongly zero-dimensional topologies on
$B$. We refer to such a pair $(B, \J)$ as a \emph{Grothendieck Boolean
  algebra} $\BJ$; and we now have notions of Grothendieck matched pair
of algebras $\BJM$ and of $\BJM$-set which, when deployed in the
theorem above, make it valid for \emph{infinitary} cartesian closed
varieties.

This concludes our overview of the paper, and the reader will notice
that we give scarcely any examples. The justification for this is the
companion paper~\cite{Garner2023CartesianII}, which begins a programme
to develop the theory of $\BJM$-sets and link them to structures from
operator algebra; in particular, we will see that any matched pair
$\BM$ has an associated topological category, and that for suitable,
and natural, choices of $\BM$, we can recover the \'etale topological
groupoids which give rise to structures such as Cuntz
$C^\ast$-algebras, Leavitt path algebras, and $C^\ast$-algebras
associated to self-similar groups, and so on.

\section{Background}
\label{sec:background}

\subsection{Conventions}
\label{sec:category-theory}

Given sets $I$ and $J$ we write $J^I$ for the set of functions
from $I$ to $J$. If $u \in J^I$, we write $u_i$ for the value of the
function $u$ at $i \in I$; on the other hand, given a family of
elements $(t_i \in J : i \in I)$, we write $\lambda i.\, t_i$ for the
corresponding element of $J^I$. Given $t \in J^I$, $i \in I$ and
$j \in J$, we may write $t[j/t_i]$ for the function which agrees with
$t$ except that its value at $i$ is given by $j$. We may identify a
natural number $n$ with the set
$\{1, \dots, n\} \subseteq \mathbb{N}$.

A category $\C$ is \emph{concrete} if it comes equipped with a
faithful functor $U$ to the category of sets. This $U$ is often an
obvious ``forgetful'' functor, in which case we suppress it from our
notation. A \emph{concrete functor} $(\C, U) \rightarrow (\C', U')$ is
a functor $H \colon \C \rightarrow \C'$ with $U' H = U$. Such an $H$
associates to each $\C$-structure on a set $X$ a corresponding
$\C'$-structure, in such a way that each $\C$-homomorphism
$f \colon X \rightarrow Y$ is also a homomorphism of the associated
$\C'$-structures. A \emph{concrete isomorphism} is an invertible
concrete functor; this amounts to a bijection between $\C$-structures
and $\C'$-structures on each set $X$ for which the homomorphisms match
up.

If $\C$ is a concrete category and $X$ a set, then a \emph{free
  $\C$-object on $X$} is a $\C$-object $\mdl[F](X)$ endowed with a
function $\eta_X \colon X \rightarrow F(X)$, the \emph{unit}, such
that, for any $\C$-object $\mdl[Y]$, each function
$f \colon X \rightarrow Y$ has a unique factorisation through $\eta_X$
via a $\C$-homomorphism
$f^\dagger \colon \mdl[F](X) \rightarrow \mdl[Y]$. We say that
\emph{free $\C$-structures exist} if they exist for every set $X$;
this is equivalent to the faithful functor $U$ having a left adjoint.

\subsection{Varieties and algebraic theories}
\label{sec:algebraic-theories}

By a \emph{variety} $\V$, we mean the concrete category of (possibly
empty) models of a (possibly infinitary) single-sorted equational
theory. For theories with a mere \emph{set} of function symbols, free
$\V$-structures always exist; we will relax this by allowing a proper
\emph{class} of function symbols, but still assuming that free
$\V$-structures exist. So the category of complete join-lattices is a
variety in our sense, but not the category of complete Boolean
algebras. We write $\Var$ for the category of varieties and concrete
functors between them.

A variety is \emph{non-degenerate} if it contains a structure with at
least two elements. To within concrete isomorphism, there are two
degenerate varieties: $\V_\mathbf{1}$ is the full subcategory of
$\cat{Set}$ on the one-element sets, while $\V_\mathbf{2}$ is the full
subcategory on the zero- and one-element sets. The former is the
category of models of the equational theory with no operation symbols
and the axiom $x = y$; while the latter is the category of models of
the theory with a single constant $c$ and the axiom $c = x$.

A given variety may be axiomatised by operations and equations in many
ways; however, there is always a maximal choice, which is captured by
the following notion of \emph{algebraic theory}. This is what a
universal algebraist would call an (infinitary) abstract clone, and
what a category theorist would call a monad relative to the identity
functor $\cat{Set} \rightarrow \cat{Set}$.

\begin{Defn}[Algebraic theories]
  \label{def:1}
  An \emph{algebraic theory} $\mathbb{T}$ comprises:
  \begin{itemize}
  \item For each set $I$, a set $T(I)$ of
    \emph{$\mathbb{T}$-operations of arity $I$};
  \item For each set $I$ and $i \in I$, an element $\pi_i \in
    T(I)$ (the \emph{$i$th projection});
  \item For all sets $I,J$ a \emph{substitution} function
    $T(I) \times T(J)^I \mapsto T(J)$, written as  $(t, u) \mapsto t(u)$, or when
    $I = n$ as $(t,u) \mapsto t(u_1, \dots, u_n)$;
  \end{itemize}
  all subject to the axioms:
  \begin{itemize}
  \item $t(\lambda i.\, \pi_i) = t$ for all $t \in T(I)$;
  \item $\pi_i(u) = u_i$ for all $u \in T(J)^I$ and $i \in I$;
  \item $(t(u))(v) = t(\lambda i.\, u_i(v))$ for all $t \in T(I)$, $u \in T(J)^I$ and $v \in T(K)^J$.
  \end{itemize}

  \noindent If $\mathbb{S}$ and $\mathbb{T}$ are algebraic theories,
  then a \emph{homomorphism of algebraic theories}
  $\varphi \colon \mathbb{S} \rightarrow \mathbb{T}$ comprises
  functions $\varphi_I \colon S(I) \rightarrow T(I)$ for each set $I$,
  such that:
  \begin{itemize}
  \item $\varphi_I(\pi_i) = \pi_i$ for all $i \in I$;
  \item $\varphi_J(t(u)) = \varphi_I(t)(\lambda i.\, \varphi_{J}(u_i))$
    for all $t \in T(I)$ and $u \in T(J)^I$.
  \end{itemize}
  We write $\AT$ for the category of algebraic theories and 
  homomomorphisms.
\end{Defn}

An algebraic theory is said to be \emph{non-degenerate} if
$\pi_1 \neq \pi_2 \in T(2)$, or equivalently, if $i \neq j \in I$
implies $\pi_i \neq \pi_j \in T(I)$. To within isomorphism, there are
exactly two degenerate algebraic theories: $\mathbb{T}_\mathbf{1}$,
in which $T_\mathbf{1}(I) = 1$ for all $I$; and
$\mathbb{T}_{\mathbf{2}}$, in which $T_\mathbf{2}(0) = 0$ and
$T_\mathbf{2}(I) = 1$ otherwise. 

When working with an algebraic theory $\mathbb{T}$ we will deploy
\emph{variable notation}. For example, in the algebraic theory of
semigroups, the defining axiom is expressed by the equality left below
in $T(3)$; however, we would prefer to write it as to the right.
\begin{equation*}
  m(m(\pi_1, \pi_2), \pi_3) = m(\pi_1, m(\pi_2, \pi_3)) \qquad m(m(x,y),z) = m(x, m(y,z))\rlap{ .}
\end{equation*}
We may do so if we view this right-hand equality as universally
quantified over all sets $I$ and all elements $x,y,z \in T(I)$. It
then implies the left-hand equality on taking
$(x,y,z) = (\pi_1, \pi_2, \pi_3)$, and conversely, is implied by the
left equality via substitution. Our convention throughout will be that
any $x$-, $y$- or $z$-symbol (possibly subscripted) appearing in an
equality is to be interpreted in this way.

\subsection{Semantics and realisation}
\label{sec:structure-semantics}

We now draw the link between algebraic theories and varieties via the
\emph{semantics} of an algebraic theory.

\begin{Defn}[Category of models of a theory]
  \label{def:2}
  A \emph{model} $\mdl[X]$ for an algebraic theory $\mathbb{T}$ is a
  set $X$ together with for each set $I$ an \emph{interpretation
    function} $T(I) \times X^{I} \rightarrow M$, written as
  $(t, a) \mapsto \dbr{t}(a)$ satisfying the following axioms:
  \begin{itemize}
  \item $\dbr{\pi_i}(x) = x_i$ for all $x \in X^I$ and $i \in I$;
  \item $\dbr{t(u)}(x) = \dbr{t}(\lambda i.\, \dbr{u_i}(x))$ for all
    $t \in T(I)$, $u \in T(J)^I$ and $x \in A^J$.
  \end{itemize}
  A \emph{$\mathbb{T}$-model homomorphism}
  $\mdl[X] \rightarrow \mdl[Y]$ is a function
  $f \colon X \rightarrow Y$ with
  $f(\dbr{t}_{\mdl[X]}(x)) = \dbr{t}_{\mdl[Y]}(f(x))$ for all
  $t \in T(I)$ and $a \in X^I$. We write $\mathbb{T}\text-\cat{Mod}$
  for the concrete category of $\mathbb{T}$-models and homomorphisms.
\end{Defn}

The category $\mathbb{T}\text-\cat{Mod}$ can be presented as the
models of an equational first-order theory, whose proper class of
function-symbols is given by the disjoint union of the $T(I)$'s.
Moreover, free $\mathbb{T}$-models exist; for indeed, given a set $X$,
the set $T(X)$ becomes a $\mathbb{T}$-model $\mdl[T](X)$ on defining
$\dbr{t}_{\mdl[T](X)}(u) = t(u)$, and now the map
$\eta_X \colon X \rightarrow T(X)$ sending $x$ to $\pi_x$ exhibits
$\mdl[T](X)$ as free on $X$. Thus, the concrete category
$\mathbb{T}\text-\cat{Mod}$ is a variety for any theory $\mathbb{T}$.
In fact, this process is functorial:

\begin{Defn}[Semantics of algebraic theories]
  \label{def:3}
  For any homomorphism
  $\varphi \colon \mathbb{S} \rightarrow \mathbb{T}$ of algebraic
  theories, we write
  $\varphi^\ast \colon \mathbb{T}\text-\cat{Mod} \rightarrow
  \mathbb{S}\text-\cat{Mod}$ for the concrete functor which to each
  $\mathbb{T}$-model $\mdl[X]$ associates the $\mathbb{S}$-model
  structure $\varphi^\ast \mdl[X]$ on $X$ with
  $\dbr{t}_{\varphi^\ast \mdl[X]}(x) = \dbr{\varphi(t)}_{\mdl[X]}(x)$.
  We write
  $(\thg)\text-\cat{Mod} \colon \AT^\mathrm{op} \rightarrow \Var$ for
  the functor sending each algebraic theory to its concrete category
  of models, and each homomorphism $\varphi$ to $\varphi^\ast$.
\end{Defn}

A basic result in the functorial semantics of algebraic theories is
that $(\thg)\text-\cat{Mod}$ is an equivalence of categories. In
particular, it is essentially surjective, which is to say that every
variety is realised by some algebraic theory; here, we say that an
algebraic theory $\mathbb{T}$ \emph{realises} a variety $\V$ if
$\mathbb{T}\text-\cat{Mod}$ and $\V$ are concretely isomorphic. For
example, the degenerate varieties $\V_\mathbf{1}$, $\V_\mathbf{2}$ are
realised by the degenerate algebraic theories $\mathbb{T}_\mathbf{1}$,
$\mathbb{T}_{\mathbf{2}}$. In general, we can find a $\mathbb{T}$
which realises a variety $\V$ using free objects in $\V$. Writing
$\mdl[T](I)$ for the free $\V$-object on $X$, with unit
$\eta_I \colon I \rightarrow T(I)$, the desired theory $\mathbb{T}$
has sets of operations $T(I)$; projection elements
$\pi_i = \eta_X(i) \in T(I)$; and substitution given by
$t(u) = u^\dagger(t)$.

\subsection{Cartesian closed varieties}
\label{sec:cart-clos-vari}
Any variety $\V$ has finite products, with the product of
$\mdl[X],\mdl[Y] \in \V$ being $X \times Y$ with the componentwise
$\V$-structure. We say $\V$ is \emph{cartesian closed} if for every
$\mdl[Y] \in \V$, the functor
$(\thg) \times \mdl[Y] \colon \V \rightarrow \V$ has a right adjoint.
More elementarily, this means that for every $\mdl[Y], \mdl[Z]$ in
$\V$, there is a ``function-space'' $\mdl[Z]^{\mdl[Y]} \in \V$ and a
homomorphism
$\mathrm{ev} \colon \mdl[Z]^{\mdl[Y]} \times \mdl[Y] \rightarrow
\mdl[Z]$ (``evaluation''), such that for all
$f \colon \mdl[X] \times \mdl[Y] \rightarrow \mdl[Z]$, there is a
unique $\bar f \colon \mdl[X] \rightarrow \mdl[Z]^{\mdl[Y]}$ with
$\mathrm{ev} \circ (\bar f \times 1) = f$. Note that, in particular,
the degenerate varieties $\V_\mathbf{1}$ and $\V_\mathbf{2}$ are
cartesian closed, since they are equivalent to the one- and
two-element Heyting algebras respectively.

The simplest possible class of non-degenerate cartesian closed
varieties are the varieties of \emph{$M$-sets} for a monoid $M$: sets
equipped with an associative, unital left $M$-action. It is well known
that the variety of $M$-sets, being a presheaf category, is cartesian
closed; we record the structure here for future reference.
\begin{Prop}
  \label{prop:1}
  The variety of $M$-sets is cartesian closed.
\end{Prop}
\begin{proof}
  For $M$-sets $Y$ and $Z$, the function-space $Z^Y$ is the set of
  $M$-set maps $\varphi \colon M \times Y \rightarrow Z$
  (where $M$ acts on itself by multiplication) under the action
  \begin{equation}
    \label{eq:1}
    m, f \qquad \mapsto \qquad m^\ast f = (\lambda n, y.\, f(nm, y))\rlap{ .}
  \end{equation}
  Evaluation
  $\mathrm{ev} \colon Z^Y \times Y \rightarrow Z$ is given by
  $\mathrm{ev}(f, y) = f(1, y)$; and given a homomorphism
  $f \colon X \times Y \rightarrow Z$, its transpose
  $\bar f \colon X \rightarrow Z^Y$ is given by
  $\bar f(x)(m, y) = f(mx, y)$.
\end{proof}

\section{Boolean algebras and $B$-sets}
\label{sec:b-sets-finitary}

\subsection{Varieties of $B$-sets}
\label{sec:b-sets}

In this section, we discuss another important class of non-degenerate
cartesian closed varieties, namely the varieties of \emph{$B$-sets}
for a Boolean algebra $B$, as introduced by Bergman
in~\cite{Bergman1991Actions}. In what follows, we write
$(\vee, \wedge, 0, 1)$ for the distributive lattice structure of a
Boolean algebra $B$, and $(\thg)'$ for its negation; we say that $B$
is \emph{non-degenerate} if $0 \neq 1$.

\begin{Defn}[Variety of $B$-sets]
  \label{def:5}
  Let $B$ be a non-degenerate Boolean algebra. A \emph{$B$-set} is a
  set $X$ endowed with an action $B \times X \times X \rightarrow X$,
  written $(b,x,y) \mapsto b(x,y)$, satisfying the axioms
  \begin{equation}
    \label{eq:2}
    \begin{aligned}
      b(x,x) &= x \qquad b(b(x,y),z) = b(x,z) \qquad b(x,b(y,z)) = b(x,z)\\
      1(x,y) &= x \qquad b'(x,y) = b(y,x) \qquad (b \wedge c)(x,y) = b(c(x,y),y)\rlap{ .}
    \end{aligned}
  \end{equation}
  We write $B\text-\cat{Set}$ for the variety of $B$-sets.
\end{Defn}

As explained in~\cite{Bergman1991Actions}, the first three axioms make
each $b(\thg, \thg) \colon X \times X \rightarrow X$ into a
\emph{decomposition
  operation}~\cite[Definition~4.32]{Mckenzie1987Algebras}, meaning
that it induces a direct product decomposition
$X \cong X_1 \times X_2$ where $X_1$ and $X_2$ are quotients of $X$ by
suitable equivalence relations. The first equivalence relation
$\equiv_b$ is defined by
\begin{equation}
  \label{eq:3}
  x \equiv_b y \iff b(x,y) = y\rlap{ ;}
\end{equation}
the second dually relates $x$ and $y$ just when
$b(x,y) = x$ but, in light of the fifth $B$-set axiom, can
equally be described as $\equiv_{b'}$. In fact, as 
in~\cite[Theorem~4.33]{Mckenzie1987Algebras}, we can recover
$b(\thg, \thg)$ from $\equiv_b$ and $\equiv_{b'}$, since $b(x,y)$ is
the unique element of $X$ with
\begin{equation}
  \label{eq:4}
  b(x,y) \equiv_b x \qquad \text{and} \qquad b(x,y) \equiv_{b'} y\rlap{ .}
\end{equation}
Thus, we can recast the notion of $B$-set in terms of a set equipped
with a suitable family of equivalence relations:

\begin{Prop}
  \label{prop:2}
  Let $B$ be a non-degenerate Boolean algebra. Each $B$-set structure
  on a set $X$ induces equivalence relations
  $(\mathord{\equiv_b} : b \in B)$ as in~\eqref{eq:3} which satisfy:
  \begin{enumerate}[(i)]
  \item If $x \equiv_b y$ and $c \leqslant b$ then $x \equiv_c y$;
  \item $x \equiv_1 y$ if and only if $x = y$, and $x \equiv_0 y$ always;
  \item If $x \equiv_b y$ and $x \equiv_c y$ then $x \equiv_{b \vee c}
    y$;
  \item For any $x,y \in X$ and $b \in B$, there is $z
    \in X$ such that $z \equiv_b x$ and $z \equiv_{b'} y$.
  \end{enumerate}
  Any family of equivalence relations $(\mathord{\equiv_b} : b \in B)$
  satisfying (i)--(iv) arises in this way from a unique $B$-set
  structure on $X$ whose operations are characterised
  by~\eqref{eq:4}. Furthermore, under this correspondence, a function
  $X \rightarrow Y$ between $B$-sets is a homomorphism if and only it
  preserves each equivalence relation $\equiv_b$.
\end{Prop}

\begin{proof}
  Given $B$-set structure on $X$, each $\equiv_b$ as in~\eqref{eq:3}
  is an equivalence relation
  by~\cite[Lemma~on~p.162]{Mckenzie1987Algebras}. To verify (i), if
  $x \equiv_b y$ and $c \leqslant b$, then
  $c(x,y) = (c \wedge b)(x,y) = c(b(x,y), y) = c(y,y) = y$, so
  $x \equiv_c y$. Next, (ii) follows immediately from $1(x,y) = x$ and
  $0(x,y) = y$. For (iii), if $b(x,y) = y$ and $c(x,y) = y$, then
  $(b \vee c)(x,y) = b(x, c(x,y)) = b(x,y) = y$. Finally, for (iv), we
  take $z = b(x,y)$; then $b(z,x) = b(b(x,y),x) = b(x,x) = x$ and
  $b'(z,y) = b(y,z) = b(y, b(x,y)) = b(y,y) = y$ as desired. We argued
  above that we can reconstruct the $B$-set operations from the
  $\equiv_b$'s, so this gives an injective map from $B$-set structures
  on $X$ to families of equivalence relations satifying (i)--(iv).

  To show surjectivity, consider a family
  $(\mathord{\equiv_b} : b \in B)$ satisfying (i)--(iv). For any
  $x,y \in X$ and $b \in B$, the element whose existence is asserted
  by (iv) is, by (ii) and (iii), \emph{unique}. If we write it as
  $b(x,y)$ as in~\eqref{eq:4}, then we claim this assignment endows
  $X$ with $B$-set structure. Indeed:
  \begin{itemize}
  \item Since $x \equiv_b x$ and $x \equiv_{b'} x$, we have
    $b(x,x) = x$;
  \item Since $b(b(x,y),z) \equiv_b b(x,y) \equiv_b x$ and
    $b(b(x,y),z) \equiv_{b'} z$, we have $b(b(x,y),z) = b(x,z)$, and
    likewise we have $b(x, b(y,z)) = b(x,z)$;
  \item Since $x$ is the \emph{only} $z$ with $x \equiv_1 z$, we have
    $1(x,y) = x$;
  \item Since $b(x,y) \equiv_{b'} y$ and $b(x,y) \equiv_{b} y$ we
    have $b'(y,x) = b(x,y)$;
  \item By (i) we have
    $b(c(x,y),y) \equiv_{b \wedge c} c(x,y) \equiv_{b \wedge c} x$.
    Similarly
    $b(c(x,y),y) \equiv_{b \wedge c'} c(x,y) \equiv_{b \wedge c'} y$,
    and also $b(c(x,y),y) \equiv_{b'} y$, whence by (iii),
    $b(c(x,y),y) \equiv_{(b \wedge c)'} y$. Thus
    $b(c(x,y),y) = (b \wedge c)(x,y)$.
  \end{itemize}
  Moreover, this $B$-set structure induces the given equivalence
  relations $\equiv_b$; indeed, since $b(x,y) \equiv_{b'} y$ and
  $b(x,y) \equiv_b x$ we have by (i)--(iii) that $b(x,y) = y$ if and
  only if $x \equiv_b y$. Finally, any $B$-set homomorphism
  $f \colon X \rightarrow Y$ clearly preserves each $\equiv_b$;
  conversely, if $f$ preserves each $\equiv_b$, then from~\eqref{eq:4}
  in $X$ we have $f(b(x,y)) \equiv_b f(x)$ and
  $f(b(x,y)) \equiv_{b'} f(y)$, and so $f(b(x,y)) = b(f(x), f(y))$
  by~\eqref{eq:4} in $Y$.
\end{proof}

\begin{Rk}
  \label{rk:5}
  Conditions (i)--(iii) above say that, for any elements $x,y$ of a
  $B$-set $X$, the set $\sheq{x}{y} = \{b \in B : x \equiv_b y\}$ is
  an \emph{ideal} of the Boolean algebra $B$; and since each
  $\equiv_b$ is an equivalence relation, the function
  $\sheq{\,\ }{\,\ } \colon X \times X \rightarrow \mathrm{Idl}(B)$ is
  an \emph{$\mathrm{Idl}(B)$-valued equivalence relation}, in the
  sense that $\sheq x x = 1$, $\sheq x y = \sheq y x$ and
  $\sheq x y \wedge \sheq y z \leqslant \sheq x z$. So $X$ becomes an
  \emph{$\mathrm{Idl}(B)$-valued set} in the sense
  of~\cite{Fourman1979Sheaves}---but one of a rather special kind,
  since in a general $\mathrm{Idl}(B)$-valued set the equality
  $\sheq{\,\ }{\,\ }$ need only be a \emph{partial} equivalence
  relation.
  As explained in~\cite{Fourman1979Sheaves}, $\mathrm{Idl}(B)$-valued
  sets are a way of presenting sheaves on $B$, and so the preceding
  observations draw the link between $B$-sets and sheaves that was
  central to~\cite{Bergman1991Actions}. In this context, the totality
  of our $\sheq{\,\ }{\,\ }$ reflects the fact that the elements of a
  $B$-set $X$ correspond to \emph{total} elements of the corresponding
  sheaf.
\end{Rk}
By exploiting Proposition~\ref{prop:2} we can now prove easily that:
\begin{Prop}
  \label{prop:3}
  The variety of $B$-sets is cartesian closed.
\end{Prop}
\begin{proof}
  Given $B$-sets $Y$ and $Z$, we consider the set $Z^Y$ of $B$-set
  homomorphisms $Y \rightarrow Z$. We claim this is a $B$-set
  under the pointwise equivalence relations $\equiv_b$. Only
  axiom (iv) is non-trivial. So suppose $f,g \in Z^Y$ and $b \in B$.
  For each $y \in Y$, we have $h(y) \in Z$ such that
  $h(y) \equiv_b f(y)$ and $h(y) \equiv_{b'} g(y)$, and so
  $h \colon Y \rightarrow Z$ will satisfy $h \equiv_b f$ and
  $h \equiv_{b'} g$ so long as it is in fact a homomorphism. So
  suppose that $y_1 \equiv_c y_2$ in $Y$; we must show
  $f(y_1) \equiv_c f(y_2)$. Since $h(y_i) \equiv_{b} f(y_i)$ and
  $f(y_1) \equiv_c f(y_2)$ (as $f$ is a homomorphism) we have by (i)
  that
  $h(y_1) \equiv_{b \wedge c} f(y_1) \equiv_{b \wedge c} f(y_2)
  \equiv_{b \wedge c} h(y_2)$; and similarly
  $h(y_1) \equiv_{b' \wedge c} h(y_2)$. Thus $h(y_1) \equiv_c h(y_2)$
  by (iii) and so $h$ is a homomorphism as desired. So $Z^Y$ is a
  $B$-set under the pointwise structure; whereupon it is clear that
  the usual evaluation map
  $\mathrm{ev} \colon Z^Y \times Y \rightarrow Z$ is a homomorphism,
  and that for any homomorphism $f \colon X \times Y \rightarrow Z$,
  the usual transpose $\bar f \colon X \rightarrow Z^Y$ is a
  homomorphism: so $Z^Y$ is a function-space as desired.
\end{proof}

\subsection{Varieties of $\BJ$-sets}
\label{sec:varieties-b_j-sets}

If $n$ is a finite set, then as
in~\cite[Proposition~4.3]{Johnstone1990Collapsed}, a $\P(n)$-set
structure on a set $X$ determines and is determined by equivalence
relations $\equiv_{\{1\}}, \dots, \equiv_{\{n\}}$ on $X$, for which
the quotient maps exhibit $X$ as the product of the sets
$X \quot \mathord{\equiv_{\{i\}}}$. Thus the category of $\P(n)$-sets
is equivalent to the category of $n$-fold cartesian products of sets.
However, this does not carry over to infinite sets $I$, for which a
$\P(I)$-set is more general than an $I$-fold cartesian product of
sets. The reason is that the notion of $\P(I)$-set does not pay regard
to the \emph{infinite} joins needed to construct each $A \subseteq I$
from atoms $\{i\}$. This can be rectified by equipping $\P(I)$ with a
suitable collection of ``well-behaved'' joins.

\begin{Defn}[Partition]
  \label{def:6}
  Let $B$ be a Boolean algebra and $b \in B$. A \emph{partition of
    $b$} is a subset $P \subseteq B \setminus \{0\}$ such that
  $\bigvee P = b$, and $c \wedge d = 0$ whenever $c \neq d \in P$. An
  \emph{extended partition of $b$} is a subset $P \subseteq B$
  (possibly containing $0$) satisfying the same conditions. If $P$ is
  an extended partition of $b$, then we write
  $P^- = P \setminus \{0\}$ for the corresponding partition. We 
  say merely ``partition'' to mean ``partition of $1$''.
\end{Defn}
\begin{Defn}[Zero-dimensional topology, Grothendieck Boolean algebra]
  \label{def:7}
  A \emph{zero-dimensional topology} on a Boolean
  algebra $B$ is a collection $\Top$ of partitions of $B$ which
  contains every finite partition, and satisfies:
  \begin{enumerate}[(i)]
  \item If $P \in \Top$, and $Q_b \in \J$ for each $b \in P$, then
    $P(Q) = \{ b \wedge c : b \in P, c \in Q_b\}^- \in \J$;
  \item If $P \in \J$ and $\alpha \colon P \rightarrow I$ is a
    surjective map, then each join $\bigvee \alpha^{-1}(i)$ exists and
    $\alpha_!(P) = \{\textstyle\bigvee \alpha^{-1}(i) : i \in I\} \in
    \Top$.
  \end{enumerate}
  A \emph{Grothendieck Boolean algebra} $\BJ$ is a Boolean algebra $B$
  with a zero-dimensional topology $\J$. A \emph{homomorphism} of
  Grothendieck Boolean algebras $f \colon \BJ \rightarrow C_\K$ is a
  Boolean homomorphism $f \colon B \rightarrow C$ such that $P \in \J$
  implies $f(P)^- \in \K$.
\end{Defn}

A zero-dimensional topology on $B$ is a special kind of Grothendieck
topology on $B$ in the sense of~\cite[\sec
II.2.11]{Johnstone1982Stone}, wherein the covers of $1 \in B$ are the
elements of $\J$, and the covers of an arbitrary $b \in B$ are given
by:
\begin{Defn}[Local partitions]
  \label{def:8}
  Let $\BJ$ be a Grothendieck Boolean algebra and $b \in B$. We write
  $\Top_b$ for the set of partitions of $b$ characterised by:
    \begin{equation*}
      P \in \J_b \iff P \cup \{b'\} \in \J \iff P \subseteq Q \in \J \text{ and } \bigvee P = b\rlap{ .}
    \end{equation*}
\end{Defn}
However, our presentation follows not~\cite{Johnstone1982Stone} but
rather~\cite{Name2013Boolean}---according to which, our Grothendieck
Boolean algebras are the ``subcomplete, locally refinable Boolean
partition algebras''. Via the general theory of~\cite[\sec
II.2.11]{Johnstone1982Stone}, any Grothendieck Boolean algebra
generates a \emph{locale} (= complete Heyting algebra) given by the
set $\mathrm{Idl}_\J(B)$ of ideals $I \subseteq B$ which are
\emph{$\J$-closed}, meaning that $b \in I$ as soon as $P \subseteq I$
for some $P \in \J_b$. The locales so arising are the \emph{strongly
  zero-dimensional locales} considered
in~\cite{Johnstone1997Cartesian}, and in fact, our category of
Grothendieck Boolean algebras is dually equivalent to the category of
strongly zero-dimensional
locales~\cite[Theorem~24]{Name2013Ultraparacompactness}.

\begin{Defn}[Variety of $\BJ$-sets]
  \label{def:9}
  Let $\BJ$ be a non-degenerate Grothendieck Boolean algebra. A
  \emph{$\BJ$-set} is a $B$-set $X$ endowed with a function
  $P \colon X^P \rightarrow X$ for each infinite $P \in \Top$,
  satisfying:
  \begin{equation}\label{eq:5}
    P(\lambda b.\, x) = x\ \ \ \ \  P(\lambda b.\, b(x_b, y_b)) = P(\lambda b.\, x_b)\ \ \ \ \ 
    b(P(x), x_b) = x_b \text{ $\forall b \in P$.}
  \end{equation}
  We write ${\BJ}\text-\cat{Set}$ for the variety of $\BJ$-sets.
\end{Defn}

Note that any non-degenerate Boolean algebra $B$ has a least
zero-dimensional topology given by the collection of all finite
partitions of $B$. In this case, $B_\J$-sets are just $B$-sets, so
that~\cref{def:9} includes~\cref{def:5} as a special case.

While the existence of functions like $P$ above looks like extra
structure on a $B$-set, it is in fact a property, rather like
the existence of
inverses in a monoid:
\begin{Prop}
  \label{prop:4}
  Let $B$ be a non-degenerate Boolean algebra and $P$ a partition of~$B$.
  \begin{enumerate}[(i)]
  \item An operation $P$ on a $B$-set $X$ satisfying the
    axioms~\eqref{eq:5} is unique if it exists.
  \item If $X$ and $Y$ are $B$-sets admitting the operation $P$, then
    any homomorphism of $B$-sets $f \colon X \rightarrow Y$ will
    preserve it.
  \end{enumerate}
\end{Prop}
\begin{proof}
  For (i), suppose $P$ and $P'$ both satisfy the axioms
  of~\eqref{eq:5}. For any $x \in X^P$ we have
  $P(x) = P(\lambda b.\, b(P'(x), x_b)) = P(\lambda b. \, P'(x)) =
  P'(x)$ and so $P = P'$. For (ii), let $x \in X^P$ again; since
  $b(P(x), x_b) = x_b$ and $f$ preserves $b$, we have
  $b(f(P(x)), f(x_b)) = f(x_b)$, and so
  \begin{equation*}
    P(\lambda b.\, f(x_b)) = P(\lambda b.\, b(f(P(x)), f(x_b))) = P(\lambda b.\, f(P(x))) = f(P(x))\rlap{ .}\qedhere
  \end{equation*}
\end{proof}

It follows that $\BJ\text-\cat{Set}$ is a full subcategory
of $B\text-\cat{Set}$. We can also characterise this subcategory in
terms of the induced equivalence relations of
Proposition~\ref{prop:2}.

\begin{Prop}
  \label{prop:5}
  Let $\BJ$ be a non-degenerate Grothendieck Boolean algebra. A
  $B$-set $X$ is a $\BJ$-set if, and only if, for each infinite $P \in
  \J$ and $x \in X^P$, there is a unique element $z \in X$ with $z
  \equiv_b x_b$ for all $b \in P$.
\end{Prop}

\begin{proof}
  First, suppose $X$ is a $\BJ$-set. Given $P \in \J$ infinite and
  $x \in X^P$, we define $z = P(x)$; we now have $z \equiv_{b} x_b$ by
  the right-hand axiom in~\eqref{eq:5}, and if $z' \equiv_b x_b$ for
  each $b \in P$, then
  $z = P(x) = P(\lambda b.\, b(x_b, z')) = P(\lambda b.\, z') = z'$ by
  the other two axioms, so that $z$ is \emph{unique} with the desired
  property.

  Suppose conversely that the stated condition holds; we endow
  $X$ with $\BJ$-set structure. Given $P \in \J$ infinite and
  $x \in X^P$, we define $P(x)$ as the unique element such that
  $P(x) \equiv_b x_b$ for all $b \in P$. Since $y \equiv_b z$ just
  when $b(y,z) = z$ we have  $b(P(x), x_b) = x_b$ for all
  $b \in P$; we also have $P(\lambda b.\, x) = x$ as $x \equiv_b x$
  for each $x$. Finally, for the second axiom in~\eqref{eq:5},
  given $b \in P$ we have
  $P(\lambda b.\, b(x_b, y_b)) \equiv_b b(x_b, y_b) \equiv_b x_b
  \equiv_b P(\lambda b.\, x_b)$ and so
  $P(\lambda b.\, b(x_b, y_b)) =P(\lambda b.\, x_b)$ by unicity.
\end{proof}

There is some awkwardness above in the distinction between infinite
and non-infinite partitions; but in fact, this can be avoided.

\begin{Prop}
  \label{prop:6}
  Let $\BJ$ be a non-degenerate Grothendieck Boolean algebra. A family
  of equivalence relations $(\mathord{\equiv_b} : b \in B)$ on a set
  $X$ determines $\BJ$-set structure on $X$ if, and only if, it
  satisfies axiom (i) of Proposition~\ref{prop:2} together with
  \begin{enumerate}
    \looseness=-1
  \item[(ii)$'$] For any $P \in \J$ and $x \in X^P$, there is a unique
    $z \in X$ with $z \equiv_{b} x_b$ for all $b \in P$.
  \end{enumerate}
\end{Prop}

\begin{proof}
  Suppose first that (i) and (ii)$'$ hold. It suffices to show that
  axioms (ii)--(iv) of Proposition~\ref{prop:2} also hold. (ii) and
  (iv) follow easily on applying (ii)$'$ to the partitions $\{1\}$ and
  $\{b, b'\}$ respectively. As for (iii), given its hypotheses, let
  $z$ be unique by (ii)$'$ such that $z \equiv_{b \vee c} x$ and
  $z \equiv_{b' \wedge c'} y$. Then using (i) and the hypotheses, we
  have $z \equiv_b x \equiv_b y$ and
  $z \equiv_{b' \wedge c} x \equiv_{b' \wedge c} y$. So $z$ agrees
  with $y$ on each part of $\{b, b' \wedge c, b' \wedge c'\}$ and so
  $z = y$; whence $y = z \equiv_{b \vee c} x$ as desired.
  
  For the converse, we must show that (i)--(iv) of
  Proposition~\ref{prop:2} imply (ii)$'$ for any \emph{finite}
  partition $P = \{b_1, \dots, b_k\}$. The unicity in (ii)$'$ holds by
  repeatedly applying axiom (iii) then axiom (i). For existence, the
  base case $k = 1$ is trivial; so let us assume the result for $k-1$,
  and prove it for $k$. By induction, we find $y$ such that
  $y \equiv_{b_1 \vee b_2} x_2$ and $y \equiv_{b_i} x_i$ for $i > 2$;
  now by (iv) we can find $z$ such that $z \equiv_{b_1} x_1$ and
  $z \equiv_{b_1'} y$. Since $b_i \leqslant b_1'$ for $i \geqslant 2$
  also $z \equiv_{b_i} y \equiv_{b_i} x_i$ for all
  $2 \leqslant i \leqslant k$, as desired.
\end{proof}

Again, using the equalities $\equiv_b$ makes it easy to show that
$\BJ$-sets are a cartesian closed variety. First we need a preparatory
lemma.

\begin{Lemma}
  \label{lem:1}
  Let $X$ be a $\BJ$-set and let $x,y \in X$.
  \begin{enumerate}[(i)]
  \item   For any $P \in \J$ and $b \in B$, if
  $x \equiv_{b \wedge c} y$ for all $c \in P$, then $x \equiv_b y$.
  \item For any $P \in \J_b$, if
  $x \equiv_{c} y$ for all $c \in P$, then $x \equiv_{b} y$.
  \end{enumerate}
\end{Lemma}

\begin{proof}
  For (i), first note
  $b \wedge P = (\{ b \wedge c : c \in P\} \cup \{b'\})^- \in \Top$ by
  applying condition (i) for a zero-dimensional topology with
  $P = \{b, b'\}$, $Q_b = P$ and $Q_{b'} = \{1\}$. Now let $z \in X$
  be unique such that $z \equiv_b y$ and $z \equiv_{b'} x$. For each
  $c \in P$ we have $z \equiv_{b \wedge c} y \equiv_{b \wedge c} x$,
  and so $z$ agrees with $x$ on each part of the partition
  $b \wedge P$ in $\J$. Thus $z = x$ and so $x = z \equiv_b y$ as
  desired. For (ii), apply (i) with $b = \bigvee P$.
\end{proof}

\begin{Rk}
  \label{rk:7}
  Note that part (ii) of this Lemma asserts that, for all elements
  $x,y$ in a $\BJ$-set $X$, the ideal
  $\sheq x y = \{b \in B : x \equiv_b y\}$ of Remark~\ref{rk:5} is a
  $\J$-closed ideal.
\end{Rk}

\begin{Prop}
  \label{prop:7}
  The variety of $\BJ$-sets is cartesian closed.
\end{Prop}

\begin{proof}
  It suffices to show that for $\BJ$-sets $Y$ and $Z$, the $B$-set
  exponential $Z^Y$ of~\cref{prop:7} is itself a $\BJ$-set. So
  suppose given a partition $P \in \J$ and a family of homomorphisms
  $f_b \colon Y \rightarrow Z$ for each $b \in P$, and define
  $g \colon Y \rightarrow Z$ by the property that
  $g(y) \equiv_b f_b(y)$ for each $b \in P$; this $g$ will be unique
  such that $g \equiv_b f_b$ for each $b \in P$, so long as it is a
  $B$-set map. So suppose that $y_1 \equiv_c y_2$; we must show that
  $g(y_1) \equiv_c g(y_2)$, for which by the preceding lemma it
  suffices to show that $g(y_1) \equiv_{b \wedge c} g(y_2)$ for all
  $b \in P$: and this follows exactly as in~\cref{prop:3}.
\end{proof}

\subsection{Theories of $B$-sets and $\BJ$-sets}
\label{sec:theories-b-sets}

To conclude this section, we describe algebraic theories which realise
the varieties of $B$-sets and $\BJ$-sets.

\begin{Defn}
  \label{def:10}
  Let $\BJ$ be a non-degenerate Grothendieck Boolean algebra. A
  \emph{$\BJ$-valued distribution} on a set $I$ is a function
  $\omega \colon I \rightarrow B$ whose restriction to
  $\mathrm{supp}(\omega) = \{i \in I : \omega(i) \neq 0\}$ is an
  injection onto a partition in $\J$.
  The \emph{theory of $\BJ$-sets} $\mathbb{T}_{\BJ}$ has $T_{\BJ}(I)$
  given by the set of $\BJ$-valued distributions on $I$; the
  projection element $\pi_i \in T_{\BJ}(I)$ given by $\pi_i(j) = 1$ if
  $i = j$ and $\pi_i(j) = 0$ otherwise; and the composition
  $\omega(\gamma) \in T_{\BJ}(J)$ of $\omega \in T_{\BJ}(I)$ and
  $\gamma \in T_{\BJ}(J)^I$ given by
  $\omega(\gamma)(j) = \textstyle\bigvee_{i \in I} \omega(i) \wedge
  \gamma_i(j)$, where this join exists using axioms (ii) and (iii) for
  a zero-dimensional topology.

  When $\J$ is the topology of finite partitions, we will write
  $\mathbb{T}_B$ in place of $\mathbb{T}_{B_\J}$ and call it the
  \emph{theory of $B$-sets}. In this case, a $B$-valued distribution
  is simply an $\omega \colon I \rightarrow B$ whose support injects
  onto a \emph{finite} partition of~$B$.
\end{Defn}

\begin{Prop}
  \label{prop:8}
  For any non-degenerate Grothendieck Boolean algebra $\BJ$ the theory
  of $\BJ$-sets realises the variety of $\BJ$-sets. In particular, for
  any non-degenerate Boolean algebra, the theory of $B$-sets realises
  the variety of $B$-sets.
\end{Prop}

\begin{proof}
  Suppose first that $X$ is a $\mathbb{T}_{\BJ}$-model. For each
  $b \in B$, we have the element $\omega_b \in T_{\BJ}(2)$ with
  $\omega_b(1) = b$ and $\omega_b(2) = b'$, while for each infinite
  partition $P \in \J$, we have the element $\omega_P \in T_{\BJ}(P)$
  given by the inclusion map $P \hookrightarrow B$. It is
  straightforward to verify that these elements satisfy the axioms
  of~\eqref{eq:2} and~\eqref{eq:5} in $\mathbb{T}_{\BJ}$, so their
  interpretations equip any $\mathbb{T}_{\BJ}$-model with the
  structure of a $\BJ$-set.

  Suppose conversely that $X$ is a $\BJ$-set, and let
  $\omega \in T_{\BJ}(I)$ and $x \in X^I$. Since $(\im \omega)^-$ is a
  partition in $\J$, we may use the $\BJ$-set structure of $X$ to
  define $\dbr{\omega}(x)$ as the unique element with
  $\dbr{\omega}(x) \equiv_{\omega(i)} x_i$ for all
  $i \in \mathrm{supp}(\omega)$. We now check the two
  $\mathbb{T}_{\BJ}$-model axioms. First, we have
  $\dbr{\pi_i}(x) \equiv_1 x_i$, i.e., $\dbr{\pi_i}(x) = x_i$. Second,
  given $\omega \in T_B(I)$ and $\gamma \in T_B(J)^I$ and $x \in X^J$,
  we have
  $\dbr{\omega}(\lambda i.\, \dbr{\gamma_i}(x))\equiv_{\omega(i)}
  \dbr{\gamma_i}(x)$ and $\dbr{\gamma_i}(x) \equiv_{\gamma_i(j)} x_j$
  for all $i \in I, j \in J$, and so
  $\dbr{\omega}(\lambda i.\, \dbr{\gamma_i}(x)) \equiv_{\omega(i)
    \wedge \gamma_i(j)} x_j$. Now~\cref{lem:1}(ii) yields
  \begin{equation*}
    \dbr{\omega}(\lambda i.\, \dbr{\gamma_i}(x)) \equiv_{\bigvee_i
      \omega(i) \wedge \gamma_i(j)} x_j \qquad \text{for all $j \in J$;}
  \end{equation*}
  but $\dbr{\omega(\gamma)}(x)$ is unique with this property, whence
  $\dbr{\omega}(\lambda i.\, \dbr{\gamma_i}(x)) =
  \dbr{\omega(\gamma)}(x)$.

  Starting from a $\BJ$-set structure on $X$, the induced
  $\mathbb{T}_{\BJ}$-model on $X$ satisfies
  $\dbr{\omega_b}(x,y) \equiv_b x$ and
  $\dbr{\omega_b}(x,y) \equiv_{b'} y$, i.e.,
  $\dbr{\omega_b}(x,y) = b(x,y)$, and also satisfies
  $\dbr{\omega_P}(x) \equiv_b x_b$ for all $b \in P$, i.e.,
  $\dbr{\omega_P}(x) = P(x)$, and so yields the original $B$-set
  structure back. Conversely, given a $\mathbb{T}_{\BJ}$-model
  structure $\dbr{\thg}$, the model structure $\dbr{\thg}'$ induced
  from the associated $\BJ$-set satisfies
  $\dbr{\omega}'(x) \equiv_{\omega(i)} x_i$, i.e.,
  $\dbr{\smash{\omega_{\omega(i)}}}(\dbr{\omega}'(x), x_i) = x_i$ for
  each $i$. But by an easy calculation, we have
  $\omega(\lambda i.\, \omega_{\omega(i)}(x_i, y_i)) = \omega(\lambda
  i.\, x_i)$ in $\mathbb{T}_B(2 \times I)$, and so
  \begin{equation*}
    \dbr{\omega}(x) = \dbr{\omega}(\lambda i.\, \dbr{\smash{\omega_{\omega(i)}}}(\dbr{\omega}'(x), x_i)) =
    \dbr{\omega}(\lambda i.\, \dbr{\omega}'(x), x_i) = \dbr{\omega}'(x)\rlap{ .}\qedhere
  \end{equation*}
\end{proof}
\begin{Rk}
  \label{rk:1}
  We can read off from this proof that the free $\BJ$-set on a set $X$
  is given by $T_{\BJ}(X)$, endowed with the $\BJ$-set structure in
  which $\omega \equiv_b \gamma$ just when
  $b \wedge \omega(x) = b \wedge \gamma(x)$ for all $x \in X$. Given a
  partition $P \in \J$ and family of elements
  $\omega \in T_{\BJ}(X)^P$, the element $P(\omega) \in T_{\BJ}(X)$ is
  given by $P(\omega)(x) = \bigvee_{b \in P} b \wedge \omega_b(x)$.
  The function $X \rightarrow T_{\BJ}(X)$ exhibiting $T_{\BJ}(X)$ as
  free on $X$ is given by $x \mapsto \pi_x$.

  As a special case, the free $\BJ$-set on two generators can
  be identified with $B$ itself (with generators $0$ and $1$) under
  the $\BJ$-set structure of ``conditioned disjunction'':
  \begin{equation*}
    b(c,d) = (b \wedge c) \vee (b' \wedge d) \qquad \text{and} \qquad P(\lambda b.\, c_b) = \textstyle\bigvee_{b \in P} b \wedge c_b\rlap{ .}
  \end{equation*}
\end{Rk}

\section{Hyperaffine theories}
\label{sec:unary-algebr-theor}

In this section, we describe, following~\cite{Johnstone1990Collapsed,
  Johnstone1997Cartesian}, the syntactic characterisation of theories
of $\BJ$-sets as (non-degenerate) \emph{hyperaffine} algebraic
theories, with the $B$-sets matching under this correspondence with
the \emph{finitary} hyperaffine theories.

As is well known, an algebraic theory is \emph{finitary} if it
corresponds to a finitary variety: which is to say that for
each $t \in T(I)$, there exist $i_1, \dots, i_n \in I$ and
$u \in T(n)$ such that $t(x) = u(x_{i_1}, \dots, x_{i_n})$. The
notion of hyperaffine algebraic theory is perhaps slightly less familiar:

\begin{Defn}[Hyperaffine operation, hyperaffine algebraic theory]
  \label{def:11}
  Let $\mathbb{T}$ be an algebraic theory. We say
  that $t \in T(I)$ is \emph{affine} if $t(\lambda i.\, x) = x$ in $T(1)$, and
  \emph{hyperaffine} if also
  $t(\lambda i.\, t(\lambda j.\, x_{ij})) = t(\lambda i.\, x_{ii})$ in
  $T(I \times I)$. We say that $\mathbb{T}$ is \emph{hyperaffine} if
  each of its operations is so. 
\end{Defn}

Our objective now is to prove:

\begin{Prop}
  \label{prop:9}
  A non-degenerate algebraic theory $\mathbb{T}$ is hyperaffine if,
  and only if, it is isomomorphic to $\mathbb{T}_{\BJ}$ for some non-degenerate
  Grothendieck Boolean algebra $\BJ$; and it is finitary and
  hyperaffine if, and only if, it is isomorphic to some $\mathbb{T}_B$.
\end{Prop}

For the ``if'' direction, we need only show that each
$\mathbb{T}_{\BJ}$ is hyperaffine, and that each $\mathbb{T}_B$ is
finitary. For the first claim, since $T_{\BJ}(1)$ is a singleton,
every operation of $\mathbb{T}_{\BJ}$ must be affine. To see that each
$\omega \in T_{\BJ}(I)$ is hyperaffine, we observe that
$\alpha = \omega(\lambda i.\, \omega(\lambda j.\, x_{ij}))$ and
$\beta = \omega(\lambda i.\, x_{ii})$ correspond to the elements of
$T_B(I \times I)$ given by, respectively,
\begin{equation*}
  \alpha(i,j) = \omega(i) \wedge \omega(j) \qquad \text{and} \qquad \beta(i,j) =
  \begin{cases}
    \omega(i) & \text{if $i = j$;}\\
    0 & \text{otherwise.}
  \end{cases}
\end{equation*}
and these are equal since $\omega$ is an injection of
$\mathrm{supp}(\omega)$ onto a partition of $B$. To show finitariness
of each $\mathbb{T}_B$, note that any $\omega \colon I \rightarrow B$
in $T_B(I)$ has \emph{finite} support $i_1, \dots, i_n$; so we have
the element $\gamma \in T(n)$ given by $\gamma(k) = \omega(i_k)$ and
see easily that $\omega = \gamma(\pi_{i_1}, \dots, \pi_{i_n})$. So
$\mathbb{T}_B$ is finitary.

The ``only if'' direction of~\cref{prop:9} is harder, and we will
attack it in stages. We begin by establishing an important property of
hyperaffine theories:
\begin{Lemma}
  \label{lem:11}
  If $\mathbb{T}$ is a hyperaffine algebraic theory, then every pair
  of operations $t \in T(I)$ and $u \in T(J)$ commute, i.e., we have
  $t(\lambda i.\, u(\lambda j.\, x_{ij})) = u(\lambda j.\, t(\lambda
  i.\, x_{ij}))$.
\end{Lemma}
\begin{proof}
  The operation $t(\lambda i.\, u(\lambda j.\, x_{ij}))$ is
  hyperaffine, which says that
  \begin{equation*}
    t(\lambda i.\, u(\lambda j.\, x_{ijij})) = 
    t(\lambda i.\, u(\lambda j.\, t(\lambda k.\, u(\lambda \ell.\, x_{ijk\ell}))))\rlap{ .}
  \end{equation*}
  Now taking $x_{ijk\ell} = y_{jk}$ gives the desired result:
  \begin{equation*}
    t(\lambda i.\, u(\lambda j.\, y_{ji})) = 
    t(\lambda i.\, u(\lambda j.\, t(\lambda k.\, u(\lambda \ell.\, y_{jk}))))
    = 
    u(\lambda j.\, t(\lambda k.\, y_{jk}))\rlap{ .}\qedhere
  \end{equation*}
\end{proof}
Now, we observed in Remark~\ref{rk:1} that, in the theory of
$\BJ$-sets, the Boolean algebra $B$ appears as the free $\BJ$-set on
two generators. This indicates how we should reconstruct a Boolean
algebra from a hyperaffine theory.
\begin{Prop}
  \label{prop:10}
  If $\mathbb{T}$ is non-degenerate hyperaffine, then $T(2)$ underlies a
  non-degenerate Boolean algebra with $1 = \pi_1$, $0 = \pi_2$, and
  $\wedge$, $\vee$ and $(\thg)'$  determined~by
  \begin{equation*}
    (b \wedge c)(x,y) = b(c(x,y),y) \quad (b \vee c)(x,y) = b(x, c(x, y)) \quad b'(x,y) = b(y,x)\rlap{ .}
  \end{equation*}
\end{Prop}

\begin{proof}
  According to~\cite[Theorem~1]{Dicker1963Set}, to give Boolean
  algebra structure on $T(2)$ is to give constants $1, 0$ and a
  ternary operation $a,b,c \mapsto a(b,c)$ (thought of as encoding the
  Boolean operation ``if $a$ then $b$ else $c$'') satisfying the five
  axioms:
  \begin{equation}
    \label{eq:6}
  \begin{aligned}
    a(b,c)(d,e) &= a(b(d,e),c(d,e)) &
    0(b,c) &= c &
    1(b,c) &= b \\
    a(0,a) &= 0 &
    a(b,0) &= b(a,0)\rlap{ ;}
  \end{aligned}
  \end{equation}
  and in this presentation, the usual Boolean operations $\wedge$,
  $\vee$ and $(\thg)'$ are re-found as $b \wedge c = b(c, 0)$ and
  $b \vee c = b(b, c)$ and $b' = b(0, 1)$. In the case of $T(2)$, if
  we take $1 = \pi_1$ and $0 = \pi_2$ and $a(b,c)$ to be the
  substitution operation in $\mathbb{T}$, then all but the
  last-displayed axiom are trivial. For this last axiom, we compute
  that
  $a(b(x,y),y) = a(b(x,y), b(y,y)) = b(a(x,y), a(y,y)) = b(a(x,y),y)$
  using affineness of $b$; commutativity of $a$ and $b$ via
  Lemma~\ref{lem:11}; and affineness of $a$. So $T(2)$ is a Boolean
  algebra structure, with operations $\wedge$, $\vee$ and $(\thg)'$ as
  displayed above; it is non-degenerate by the assumption that
  $\pi_1 \neq \pi_2 \in T(2)$.
\end{proof}

We now explain how to endow the Boolean algebra of this proposition
with a zero-dimensional topology. In this we follow~\cite[\sec
2]{Johnstone1997Cartesian} by first introducing \emph{binary reducts}
and proving some important facts about them.
\begin{Defn}[Binary reducts]
  \label{def:12}
  Let $\mathbb{T}$ be an algebraic theory and $t \in T(I)$. For each
  subset $U \subseteq I$, we write $t^{(U)} \in T(2)$ for the binary
  operation with
  \begin{equation*}
    t^{(U)}(x,y) = t\biggl(\lambda i.\,
    \begin{cases}
      x & \text{ if $i \in U$;}\\
      y & \text{ otherwise.}
    \end{cases}
    \biggr)
  \end{equation*}
  When $U$ is a singleton $\{i\}$, we may write $t^{(i)}$ rather than
  $t^{(\{i\})}$.
\end{Defn}

\begin{Lemma}
  \label{prop:11}
  Let $\mathbb{T}$ be a non-degenerate hyperaffine algebraic theory,
  let $B = T(2)$ be the Boolean algebra of~\cref{prop:10}, and let
  $t,u \in T(I)$.
  \begin{enumerate}[(i)]
  \item If $t^{(i)} = u^{(i)}$ for all $i \in I$ then $t = u$;
  \item If $i \neq j$ then $t^{(i)} \wedge t^{(j)} = 0$ in $B$;
  \item For all $U \subseteq I$, we have $t^{(U)} = \bigvee_{i \in U}
    t^{(i)}$ in $B$.
  \end{enumerate}
\end{Lemma}

\begin{proof}
  The elements of $B = T(2)$ easily satisfy the axioms of~\eqref{eq:2}
  in $\mathbb{T}$ and so via their action by substitution endow each
  $T(J)$ with a $B$-set structure. We claim that, with respect to this
  structure, we have for each $h \in T(I)$ that:
  \begin{equation}\label{eq:7}
    h(x) \text{ is unique such that } 
    h(x) \equiv_{h^{(i)}} x_i \text{ for all $i \in I$.}
  \end{equation}
  To see that $h$ satisfies the displayed condition, fix $i \in I$ and
  define $w_{jk} = x_k$ if $j=i$ and $w_{jk} = x_i$ otherwise; then we
  have
  $h^{(i)}(h(x), x_i) = h^{(i)}(h(\lambda k.\, x_k), h(\lambda k.\,
  x_i)) = h(\lambda j.\, h(\lambda k.\, w_{jk})) = h(\lambda j.\,
  w_{jj}) = x_i$ as required. To show the unicity, suppose
  $h^{(i)}(k(x),x_i) = x_i$ for each $i$, and set $z_{ij} = k(x)$ if
  $j = i$ and $z_{ij} = x_i$ otherwise; then
  $h(x) = h(\lambda i.\, h^{(i)}(k(x),x_i)) = h(\lambda i.\, h(\lambda
  j.\, z_{ij})) = h(\lambda i.\, z_{ii}) = k(x)$ as desired.

  We now prove (i)--(iii). For (i), if $t^{(i)} = u^{(i)}$ for each
  $i$, then $u(x) \equiv_{t^{(i)}} x_i$ for each $i$, whence
  $u(x) = t(x)$ by unicity in~\eqref{eq:7}. For (ii),~\eqref{eq:7}
  yields $h^{(j)}(x,y) \equiv_{h^{(i)}} y$ for $i \neq j$, i.e.,
  $y = h^{(i)}(h^{(j)}(x,y), y)$ which says
  $h^{(i)} \wedge h^{(j)} = 0$. Finally for (iii), observe
  that~\eqref{eq:7} implies that $h(x) = h(y)$ if and only if
  $x_i \equiv_{h^{(i)}} y_i$ for all $i \in I$. Thus, for any
  $t \in T(I)$, $U \subseteq I$ and $u \in T(2)$, we have
  $t^{(U)} \leqslant u$ in $T(2)$ just when
  $t^{(U)}(u(x,y), y) = t^{(U)}(x,y)$, just when
  $u(x,y) \equiv_{t^{(i)}} x$ for all $i \in U$. In particular,
  $t^{(U)} \leqslant u$ if and only if $t^{(i)} \leqslant u$ for all
  $i \in U$, so $t^{(U)} = \bigvee_{i \in U} t^{(i)}$ as desired.
\end{proof}

This lemma implies, in particular, that if $\mathbb{T}$ is hyperaffine
and $h \in T(I)$, then the set $P = \{h^{(i)} : i \in I\}^-$ is a
partition of $B = T(2)$. In this situation, we will say that the
operation $h$ \emph{realises} the partition $P$.

\begin{Prop}
  \label{prop:12}
  Let $\mathbb{T}$ be a non-degenerate hyperaffine theory and
  $B = T(2)$ the Boolean algebra of~\cref{prop:10}. The set $\J$ of all
  partitions realised by operations of $\mathbb{T}$ constitutes a
  zero-dimensional topology on $B$. If $\mathbb{T}$ is finitary, then
  $\J$ is necessarily the topology of finite partitions.
\end{Prop}
\begin{proof}
  We first show that any $P \in \J$ has a \emph{canonical realisation}
  by a (necessarily unique) $h \in T(P)$ with $h^{(b)} = b$ for all
  $b \in P$. To this end, let $k \in T(I)$ be \emph{any} realiser for
  $P$, pick an arbitrary element $b \in P$, and define a function
  $f \colon I \rightarrow P$ by taking $f(i) = k^{(i)}$ if
  $k^{(i)} \neq 0$ and $f(i) = b$ otherwise; we now easily see that
  $h(x) = k(\lambda i.\, x_{f(i)})$ is the desired canonical
  realisation for $P$.

  We now verify the axioms for a zero-dimensional topology. First
  observe that the trivial partition $\{1\}$ is (canonically) realised
  by $\pi_1 \in T(1)$; and that if all $(n-1)$-fold partitions are
  realised, then so is every $n$-fold partition $\{b_1, \dots, b_n\}$:
  for indeed, if $h \in T(n-1)$ realises
  $\{b_1 \vee b_2, \dots, b_n\}$, then
  $k(x_1, \dots, x_n) = b_1(x_1, h(x_2, \dots, x_{n}))$ is easily seen
  to realise $\{b_1, \dots, b_n\}$. So all finite partitions are in
  $\J$.

  Now for (i), suppose given $P \in \J$ and $Q_b \in \J$ for each
  $b \in B$. Let $h \in T(P)$ and $k_b \in T(Q_b)$ be their canonical
  realisers, and consider the term $\ell \in T(\sum_b Q_b)$ with
  $\ell(x) = h(\lambda b.\, k_b(\lambda c.\, x_{bc}))\rlap{ .}$ Easily
  we have
  $\ell^{(b,c)}(x,y) = h^{(b)}(k_b^{(c)}(x,y),y) = b(c(x,y),y) =
  (b\wedge c)(x,y)$ so that $\ell$ realises the partition $P(Q)$ as
  desired. Finally for (ii), let $P \in \J$ with canonical realiser
  $h \in T(P)$, and let $\alpha \colon P \rightarrow I$ be a
  surjection. Let $k(x) = h(\lambda i.\, x_{\alpha(i)})$ in $T(I)$;
  then by~\cref{prop:11}(iii), we have
  $k^{(i)} = h^{(\alpha^{-1}(i))} = \bigvee_{b \in
    \alpha^{-1}(i)}h^{(b)} = \bigvee \alpha^{-1}(i)$, so that $k$
  realises the partition $\alpha_!(P)$.

  Finally, if $\mathbb{T}$ is finitary, then we can write any $h \in
  T(I)$ as $k(x_{i_1}, \dots, x_{i_k})$ for some finite list $i_1,
  \dots, i_k \in I$. It follows that $h^{(i)} = 0$ unless $i \in
  \{i_1, \dots, i_k\}$, so that partitions which $\mathbb{T}$
  realises are precisely the finite ones.
\end{proof}

The following result now completes the proof of \cref{prop:9}.

\begin{Prop}
  \label{prop:13}
  Let $\mathbb{T}$ be a non-degenerate hyperaffine theory, and $\BJ$
  the non-degenerate Grothendieck Boolean algebra of
  \cref{prop:10,prop:12}. The maps
  \begin{equation}\label{eq:8}
    \omega_{(\thg)} \colon T(I) \rightarrow T_{\BJ}(I) \qquad t \mapsto \lambda i.\, t^{(i)}
  \end{equation}
  are the components of an isomorphism of algebraic theories
  $\mathbb{T} \cong \mathbb{T}_{\BJ}$. In particular, if $\mathbb{T}$ is finitary,
  then we have an isomorphism $\mathbb{T} \cong \mathbb{T}_B$.
\end{Prop}

\begin{proof}
  By \cref{prop:11} and definition of $\J$, each
  $\omega_t \colon I \rightarrow B$ is injective from its support onto
  a partition in $\J$, so that the maps in~\eqref{eq:8} are
  well-defined.
  
  For any $i \in I$ it is clear that
  $\omega_{\pi_i} = \pi_i \in T_{\BJ}(I)$. As for preservation of
  composition, let $t \in T(I)$ and $u \in T(J)^I$; we must show
  $t(u)^{(j)} = \bigvee_{i \in I} t^{(i)} \wedge {u_i}^{(j)}$ for each
  $j \in J$. To this end, note that the term
  $v = t(\lambda i.\, u_i(\lambda j.\, x_{ij}))$ satisfies
  \begin{equation*}
    v^{(i,j)}(x,y) = t^{(i)}(u^{(j)}(x,y),y) = (t^{(i)} \wedge {u_i}^{(j)})(x,y)\rlap{ ;}
  \end{equation*}
  whence, by~\cref{prop:11}(iii),
  $t(u)^{(j)} =
  v^{(I \times \{j\})} = \bigvee_{i \in I} \omega(i) \wedge
  \gamma_i(j)$ as desired.

  So we have a theory morphism
  $\omega_{(\thg)} \colon \mathbb{T} \rightarrow \mathbb{T}_{\BJ}$,
  whose components are injective by~\cref{prop:11}(i). It remains to
  show they are also surjective. To this end, let
  $\omega \colon I \rightarrow B$ be a distribution. By assumption,
  $\res \omega {\mathrm{supp}(\omega)}$ is an injection onto some
  $P \in \J$. So let $\iota \colon P \rightarrow I$ be the injective
  function sending $\omega(i)$ to $i$, let $h \in T(P)$ be the
  canonical realiser of $P$, and define $t \in T(I)$ to be
  $t(x) = h(\lambda b.\, \smash{x_{\iota(b)}})$. It is now clear that
  $t^{(i)} = h^{(\omega(i))} = \omega(i)$ for all $i \in
  \mathrm{supp}(\omega)$, and that $t^{(i)} = 0$ for all $i \notin
  \mathrm{supp}(\omega)$. Thus $\omega_t = \omega$ as desired.
\end{proof}

We now make the correspondences between non-degenerate Grothendieck
Boolean algebras, varieties of $\BJ$-sets, and non-degenerate
hyperaffine algebraic theories---and their finitary variants---into
functorial equivalences. Let us write:
\begin{itemize}[itemsep=0.25\baselineskip]
\item $\BA$ (resp., $\GBA$) for the category of non-degenerate
  Boolean (resp., Grothendieck Boolean)
  algebras and their homomorphisms.
\item $\HA$ (resp., $\HA^\omega$) for the full subcategory of $\AT$ on
  the non-degenerate hyperaffine (resp., finitary hyperaffine)
  algebraic theories.
\item $B\text-\Var$ (resp.,~$\BJ\text-\Var$) for the full
  subcategory of $\Var$ on the varieties isomorphic to some
  $B\text-\cat{Set}$ (resp.~$\BJ\text-\cat{Set}$).
\end{itemize}

The assignments $\BJ \mapsto \BJ\text-\cat{Set}$ and
$\BJ \mapsto \mathbb{T}_{\BJ}$ can now be made functorial. A
homomorphism of non-degenerate Grothendieck Boolean algebras
$f \colon \BJ \rightarrow B'_{\J'}$ induces, on the one hand, a
concrete functor
$f^\ast \colon B'_{\J'}\text-\cat{Set} \rightarrow
\BJ\text-\cat{Set}$, where $f^\ast$ assigns to a $B'_{\J'}$-set
structure on $X$ the $\BJ$-set structure with
$b,x,y \mapsto (fb)(x,y)$; and on the other hand, a theory
homomorphism
$\mathbb{T}_f \colon \mathbb{T}_{\BJ} \rightarrow
\mathbb{T}_{B'_{\J'}}$ with components
\begin{equation}\label{eq:9}
  (\mathbb{T}_f)_I \colon T_{\BJ}(I) \rightarrow T_{B'_{\J'}}(I) \qquad \omega \mapsto f \circ \omega\rlap{ .}
\end{equation}
In this way, we obtain functors $\mathbb{T}_{(\thg)}$ and
$(\thg)\text-\cat{Set}$ as in the statement of:

\begin{Thm}
  \label{thm:1}
  We have a triangle of equivalences, commuting to within natural
  isomorphism, as to the left in:
  \begin{equation}\label{eq:10}
    \cd[@!C@C-2.7em]{
      \GBA \ar[dr]_-{(\thg)\text-\cat{Set}} \ar[rr]^-{\mathbb{T}_{(\thg)}} & & 
      \HA \ar[dl]^-{(\thg)\text-\cat{Mod}} \\
      & (\BJ\text-\cat{Var})^\mathrm{op}
    } \qquad
    \cd[@!C@C-2.7em]{
      \BA \ar[dr]_-{(\thg)\text-\cat{Set}} \ar[rr]^-{\mathbb{T}_{(\thg)}} & & 
      \HA^\omega \ar[dl]^-{(\thg)\text-\cat{Mod}} \\
      & (B\text-\cat{Var})^\mathrm{op}
    }
  \end{equation}
  which restricts back to a triangle of equivalences as to the right.
\end{Thm}

\begin{proof}
  By \cref{prop:8,prop:9}, the equivalence
  $(\thg)\text-\cat{Mod} \colon \AT \rightarrow \Var^\mathrm{op}$
  restricts to one $\HA \rightarrow (\BJ\text-\cat{Var})^\mathrm{op}$
  and further back to one
  $\HA^\omega \rightarrow (B\text-\cat{Var})^\mathrm{op}$. Again
  because of \cref{prop:8}, the triangles commute to within
  isomorphism. So to complete the proof, it suffices to show that
  $\mathbb{T}_{(\thg)}$ is an equivalence both to the left and the
  right. We already know by \cref{prop:9} that in both cases it is
  essentially surjective, so we just need to check it is also full and
  faithful.

  As in the proof of \cref{prop:8}, given $b \in B$ we write
  $\omega_b \in T_{\BJ}(2)$ for the element with $\omega_b(1) = b$ and
  $\omega_b(2) = b'$, and given $P \in \J$ we write
  $\omega_P \in T_{\BJ}(P)$ for the element given by the inclusion
  $P \hookrightarrow B$. Note that \emph{every} element of $T_B(2)$ is
  of the form $\omega_b$ for a unique $b \in B$; so given a theory
  homomorphism $\varphi \colon \mathbb{T}_B \rightarrow \mathbb{T}_C$
  there is a unique map $f \colon B \rightarrow C$ such that
  $\varphi(\omega_b) = \omega_{f(b)}$ for each $b \in B$. Since
  $\omega_1 = \pi_1$, $\omega_{b'} = \omega_b(\pi_2, \pi_1)$ and
  $\omega_{b \wedge c} = \omega_b(\omega_c(\pi_1, \pi_2) \pi_2)$ in
  $T_B(2)$, and $\varphi$ preserves these identities, this $f$ is a
  Boolean homomorphism. Moreover, for each $\omega \in T_B(I)$ and
  $i \in I$, we have $\smash{\omega_{\omega(i)}} = \omega^{(i)} \in T_B(2)$,
  and so
  $\varphi(\omega)^{(i)} = \varphi(\omega^{(i)}) =
  \varphi(\omega_{\omega(i)}) = \omega_{f(\omega(i))}$, so that
  $\varphi(\omega) = \lambda i.\, f(\omega(i))$. In particular, $f$
  must be a homomorphism of Grothendieck Boolean algebras: for indeed,
  if $P \in \J$, then $(\im \varphi(\omega_P))^- = f(P)^- \in \K$.
  Moreover, we have that $\varphi = \mathbb{T}_f$; now if also
  $\varphi = \mathbb{T}_g$, then
  $\omega_{f(b)} = \varphi(\omega_b) = \omega_{g(b)}$ for all
  $b \in B$ and so $f = g$. So $\mathbb{T}_{(\thg)}$ is full and
  faithful as claimed, and this completes the proof.
\end{proof}

\section{Hyperaffine--unary theories}
\label{sec:gener-case:synt}

In this section, we prove our first main result, which gives a
syntactic characterisation of the algebraic theories which correspond
to cartesian closed varieties. This simplifies an existing
characterisation due to Johnstone in~\cite{Johnstone1990Collapsed}; as
such, we begin by recalling Johnstone's result, and then use it to
deduce ours.

\begin{Not}[Placed equality, dependency; \cite{Johnstone1990Collapsed}]
  Let $\mathbb{T}$ be an algebraic theory, let $q \in T(I)$ and let $i
  \in I$. Given $t,u \in
  T(J)$, we write $t
  \equiv_{q,i} u$ (read as ``$t$ and $u$ are equal in the $i$th place
  of $q$'') as an abbreviation for the assertion that
  \begin{equation*}
    q(\, x[t(y)/x_i]\,) = q(\,x[u(y)/x_i]\,)\rlap{ .}
  \end{equation*}
  We say that $q \in T(I)$ \emph{does not depend on} $i \in I$ if
  $x \equiv_{q, i} y$.
\end{Not}

\begin{Thm}[\cite{Johnstone1990Collapsed}]
  \label{thm:2}
  A non-degenerate algebraic theory $\mathbb{T}$ presents a cartesian
  closed variety if, and only if, the following two conditions hold:
  \begin{enumerate}[(i)]
  \item For every $p \in T(A)$, there exist $q \in T(B)$, families
    $u,v \in T(1)^B$, and a function $\alpha \colon B \rightarrow A$
    such that
    \begin{equation}\label{eq:11}
      q(\lambda b.\, u_b(x)) = x \quad \text{and} \quad
      u_b(p(\lambda a.\, x_a)) \equiv_{q,b} v_b(x_{\alpha(b)})  \text{ for
        all $b \in B$.}
    \end{equation}
  \item For any $q \in T(B)$, $u \in T(1)^B$ and
    $\alpha \colon B \rightarrow A$, if
    $q(\lambda b.\, u_b(x_{\alpha(b)})) \in T(A)$ does not depend on
    $i$, then $q$ does not depend on any $j \in \alpha^{-1}(i)$.
  \end{enumerate}
\end{Thm}

Our improved characterisation says that $\mathbb{T}$ presents a
cartesian closed variety if, and only if, each operation decomposes
uniquely into hyperaffine and unary parts.
\begin{Defn}[Hyperaffine--unary decomposition, hyperaffine--unary
  theory]
  \label{def:13}
  Let $\mathbb{T}$ be an algebraic theory. Given a hyperaffine
  $h \in T(I)$ and a unary $m \in T(1)$, we write $\dc h m$ for the
  operation $h(\lambda i.\, m) \in T(I)$. A \emph{hyperaffine--unary
    decomposition} of $t \in T(I)$ is a choice of $h$ and $m$ as above
  such that $t = \dc h m$. We say that the theory $\mathbb{T}$ is
  \emph{hyperaffine--unary} if every operation $t \in T(I)$ admits a
  unique hyperaffine--unary decomposition.
\end{Defn}

The crucial lemma which will enable us to prove this is:

\begin{Lemma}
  \label{lem:2}
  If the algebraic theory $\mathbb{T}$ satisfies condition (i) of
  \cref{thm:2}, then any affine operation of $\mathbb{T}$ is
  hyperaffine.
\end{Lemma}

\begin{proof}
  Let $p \in T(A)$ be affine, and let $q,u,v,\alpha$ be as in
  \cref{thm:2}(i). Note first that, since $p$ is affine, on
  substituting $x$ for each $x_a$ in the right-hand equation
  of~\eqref{eq:11}, we have that $u_b \equiv_{q,b} v_b$ for all
  $b \in B$. To show $p$ is hyperaffine, it suffices, by the left
  equation of~\eqref{eq:11}, to prove that
  $u_b(p(\lambda a.\, p(\lambda a'.\, x_{aa'}))) \equiv_{q,b}
  u_b(p(\lambda a.\, x_{aa}))$ for all $b \in B$. But we calculate
  that
  \begin{align*}
    u_b(p(\lambda a.\, p(\lambda
    a'.\, x_{aa'}))) & \equiv_{q,b} v_b(p(\lambda a'.\,
    x_{\alpha(b),a'}))
    \equiv_{q,b} u_b(p(\lambda a'.\, x_{\alpha(b),a'}))\\
    & \equiv_{q,b} v_b(x_{\alpha(b),\alpha(b)})
    \equiv_{q,b} u_b(p(\lambda a.\, x_{aa}))
  \end{align*}
  using the right equation of~\eqref{eq:11} three times and the fact that
  $u_b \equiv_{q,b} v_b$ once.
\end{proof}

\begin{Thm}
  \label{thm:3}
  An algebraic theory $\mathbb{T}$ presents a cartesian closed variety
  if, and only if, it is hyperaffine--unary.
\end{Thm}

\begin{proof}
  Firstly, if $\mathbb{T}$ is degenerate then it is both cartesian
  closed and hyperaffine, so \emph{a fortiori} hyperaffine--unary.
  Thus, we may assume henceforth that $\mathbb{T}$ is non-degenerate
  and so apply Johnstone's characterisation theorem.
  
  We first prove the \textbf{only if} direction. To begin with, we
  show that each $p \in T(A)$ has some hyperaffine--unary
  decomposition. To this end, let $q,u,v,\alpha$ be as in
  \cref{thm:2}(i). Let $h \in T(A)$ be given by
  $h(x) = q(\lambda b.\, u_b(x_{\alpha(b)}))$. By the left condition
  of~\eqref{eq:11} $h$ is affine, and so it is hyperaffine by
  \cref{lem:2}. Let $m \in T(1)$ be given by
  $m(x) = p(\lambda a.\, x)$. We now calculate that
  \begin{align*}
    p(x) &= q(\lambda b.\, u_b (p(x))) 
    = q(\lambda b.\, v_b(x_{\alpha(b)})) 
    = q(\lambda b.\, u_b(p(\lambda a.\, x_{\alpha(b)}))) \\
    &= q(\lambda b.\, u_b(m(x_{\alpha(b)}))) = h(\lambda a.\, m(x_a))
  \end{align*}
  using in succession, the left equality in~\eqref{eq:11}; the right
  equality twice; the definition of $m$; and the definition of $h$.

  We now show this decomposition of $p$ is unique. Suppose we have
  $h,h' \in T(A)$ hyperaffine and $m,m' \in T(1)$ with
  $\dc h m = \dc {h'} {m'}$. As $h$ and $h'$ are affine, we have
  $m(x) = h(\lambda a.\, m(x)) = h'(\lambda a.\, m'(x)) = m'(x)$. We
  must show also that $h = h'$. Note that
  $h(\lambda a.\, m(x_{aa})) = h(\lambda a.\, h(\lambda a' .\,
  m(x_{aa'}))) = h(\lambda a.\, h'(\lambda a' .\, m(x_{aa'})))$, so
  that the right-hand side does not depend on $x_{aa'}$ whenever
  $a \neq a'$. Thus by \cref{thm:2}(ii), we conclude that
  $h(\lambda a.\, h'(\lambda a' .\, x_{aa'}))$ does not depend on
  $x_{aa'}$ for any $a \neq a'$. It follows that
  $h(\lambda a.\, h'(\lambda a' .\, x_{aa'})) = h(\lambda a.\,
  h'(\lambda a' .\, x_{aa})) = h(\lambda a.\, x_{aa})$ and so taking
  $x_{aa'} = y_a$, we conclude that
  $h'(\lambda a.\, y_{a}) = h(\lambda a.\, h'(\lambda a.\, y_{a})) =
  h(\lambda a.\, y_a)$ so that $h = h'$ as desired.

  We now prove the \textbf{if} direction. Supposing that every
  operation of $\mathbb{T}$ has a unique hyperaffine--unary
  decomposition, we prove (i) and (ii) of Theorem~\ref{thm:2}. For
  (i), given $p \in T(A)$ with decomposition $p = \dc h m$, we take
  $q=h$, $u_a = \mathrm{id}$, $v_a = m$ and $\alpha = \mathrm{id}$ to
  obtain the required data satisfying~\eqref{eq:11}. It remains to
  verify condition (ii). So let $q \in T(B)$, $u \in T(1)^B$ and
  $\alpha \colon B \rightarrow A$ be such that $p \in T(A)$ given by
  $p(x) = q(\lambda b.\, u_b(x_{\alpha(b)}))$ does not depend on
  $x_a$. Writing $q = \dc h m$, we have
  \begin{align*}
    p(x) &= \dc h m(\lambda b.\, u_b(x_{\alpha(b)}))
    = h(\lambda b.\, m(u_b(x_{\alpha(b)}))) \\
    &= h(\lambda b.\, h(\lambda b'.\, m(u_{b'}(x_{\alpha(b)})))) =
    h(\lambda b.\, n(x_{\alpha(b)})) = 
    k(\lambda a.\, n(x_a))
  \end{align*}
  where we define $n(x) = h(\lambda b'.\, m(u_{b'}(x)))$ and
  $k(x) = h(\lambda b.\, x_{\alpha(b)})$. Now consider the hyperaffine
  operations $k',k'' \in T(A+1)$ with
  \begin{equation*}
    k'(x, y) = k(x) \qquad \text{and} \qquad 
    k''(x, y) = k(x[y/x_a])\rlap{ .}
  \end{equation*}
  Because $p = \dc k n$ does not depend on $x_a$, the operations
  $\dc {k'} n$ and $\dc {k''} n$ are equal. By unicity of
  decompositions, we have $k' = k''$ in $T(A+1)$, so that $k(x)$ does
  not depend on $x_a$. We claim it follows that $h(y)$ does not depend
  on $y_b$ for any $b \in \alpha^{-1}(a)$. For indeed, we have that
  $x = k(\lambda i.\, x) = k^{(a)}(y,x) = h^{(\alpha^{-1}(a))}(y,x)$
  in $T(2)$, since $k(x)$ does not depend on $x_a$. Since, by
  hyperaffineness of $h$, we have $h(x) \equiv_{h,b} x_b$ for each
  $b \in B$, we conclude that, for any $b \in \alpha^{-1}(a)$, we have
  $x = h^{(\alpha^{-1}(a))}(y,x) \equiv_{h,b} y$ so that $h$ does not
  depend on $b$.
\end{proof}
The proof of Theorem~\ref{thm:3} just given provides an intellectually
honest account of the genesis of the ideas in this paper; but for good
measure, we will also give a direct proof of the theorem which does
not rely on~\cite{Johnstone1990Collapsed}. For now we only
show that every cartesian closed variety is hyperaffine--unary; we
will close the loop in Section~\ref{sec:general-case:-two} once we
have a better handle on what hyperaffine--unary theories are.

It will in fact be clearer if we prove something more general. Recall
that the \emph{copower} $I \cdot X$ of some $X \in \C$ by a set $I$ is
a coproduct of $I$ copies of $X$. We write
$\iota_i \colon X \rightarrow I \cdot X$ for the coproduct
coprojections, and given maps
$(f_i \colon X \rightarrow Y)_{i \in I}$, write
$\spn{f_i} \colon I \cdot X \rightarrow Y$ for the unique map with
$\spn{f_i} \circ \iota_i = f_i$ for each $i$.

\begin{Defn}[Complete theory of dual operations]
  \label{def:22}
  Let $\C$ be a category with all set-indexed copowers and let
  $X \in \C$. The \emph{complete theory of dual operations of $X$}
  $\mathbb{T}_X$ is the algebraic theory with:
  \begin{itemize}
  \item $T_X(I) = \C(X, I \cdot  X)$;
  \item Projection elements $\iota_i \in T(I)$;
  \item Substitution $\C(X, I \cdot X) \times \C(X, J  \cdot X)^I
    \rightarrow  \C(X, J \cdot X)$  given  by
    \begin{equation*}
      (t, u)  \mapsto X \xrightarrow{t} I \cdot X \xrightarrow{\spn{u_i}} J \cdot X\rlap{ .}
    \end{equation*}
  \end{itemize}
\end{Defn}
To a universal algebraist, this is the (infinitary) \emph{clone of
  co-operations} of $X$~\cite{Csakany1985Completeness}; to a category
theorist, it is the \emph{structure} of the hom-functor
$\C(X, \thg) \colon \C \rightarrow
\cat{Set}$~\cite{Lawvere1963Functorial}. The following standard result
is now~\cite[Theorem~III.2]{Lawvere1963Functorial}:

\begin{Prop}
  \label{prop:46}
  Let $\C$ be a category with set-indexed copowers and let $X \in \C$.
  For any $Y \in \C$, the set $\C(X,Y)$ is a model for $\mathbb{T}_X$
  with
  $\dbr{t}(f) = \spn{f_i} \circ t \colon X \rightarrow I \cdot X
  \rightarrow Y$ for each $t \in T(I)$. For any
  $g \colon Y \rightarrow Z$, the function
  $g \circ (\thg) \colon \C(X,Y) \rightarrow \C(X,Z)$ is a
  $\mathbb{T}_X$-homomorphism, and so we induce a factorisation
  \begin{equation}
    \label{eq:34}
    \cd[@!C@C-3em]{
      {\C} \ar@{-->}[rr]^-{K} \ar[dr]_-{\C(X, \thg)} & &
      {\mathbb{T}_X\text-\cat{Mod}\rlap{ ,}} \ar[dl]^-{U} \\ &
      {\cat{Set}}
    }
  \end{equation}
  which is universal among factorisations of $\C(X, \thg)$ through a
  variety.
\end{Prop}
When, in the above proposition, we take $\C$ itself to be a variety
$\mathbb{T}\text-\cat{Mod}$, and take $X = \mathbf{T}(1)$, the free
model on one generator, it is visibly the case that
$\mathbb{T}_X \cong \mathbb{T}$ and that $K$ is an \emph{isomorphism}.
Thus, the fact that any cartesian closed variety is hyperaffine--unary
follows from:

\begin{Prop}
  \label{prop:47}
  Let $\C$ be a category with finite products and set-indexed
  copowers, and suppose that each
  $(\thg) \times X \colon \C \rightarrow \C$ preserves copowers; in
  particular, this will be so if $\C$ is cartesian closed. For any
  $X \in \C$, the complete theory of dual operations $\mathbb{T}_X$ is
  hyperaffine--unary.
\end{Prop}
\begin{proof}
  Since $(\thg) \times X$ preserves copowers, we may realise the
  copower $I \cdot X$ as the product $(I \cdot 1) \times X$ via the
  coprojection maps
  \begin{equation*}
    X \xrightarrow{\cong} 1 \times X \xrightarrow{\iota_i \times X} (I \cdot 1)
    \times X\rlap{ .}
  \end{equation*}
  Thus, we may write operations in $T(I)$ as pairs
  $(h,m) \colon X \rightarrow (I \cdot 1) \times X$ where
  $h \colon X \rightarrow I \cdot 1$ and $m \colon X \rightarrow X$.
  From the definition of substitution in $\mathbb{T}_X$, such an
  operation is \emph{affine} just when
  \begin{equation*}
    X \xrightarrow{(h,m)} (I \cdot 1) \times X \xrightarrow{\pi_2} X \quad = \quad  X \xrightarrow{\mathrm{id}} X\rlap{ ,}
  \end{equation*}
  i.e., just when $m = \mathrm{id}$. We claim that any such $t = (h,
  \mathrm{id})$ is in fact \emph{hyperaffine}. This follows from
  commutativity in:
  \begin{equation*}
    \cd[@!C@C-0em]{
      &
      (I \cdot 1) \times X
      \ar@/^6pt/[dr]^-{(I \cdot 1) \times (h,\mathrm{id})} \\
      X \ar[d]_-{(h,\mathrm{id})} 
      \ar@/^6pt/[ur]^-{(h,\mathrm{id})}
      \ar[r]^-{(\Delta, \mathrm{id})}&
      X \times X \times X \ar[r]^-{h \times h \times X}& 
      (I \cdot 1) \times (I \cdot 1) \times X \ar[d]^-{\theta \times X}
      \\
      (I \cdot 1) \times X \ar[rr]_-{(\Delta \cdot 1) \times X} \ar[urr]_-{\Delta \times X} & &
      ((I \times I) \cdot 1) \times X
    }
  \end{equation*}
  whose uppermost composite is the interpretation of
  $t(\lambda i.\, t(\lambda j.\, x_{ij}))$ and whose lower composite
  interprets $t(\lambda i. x_{ii})$; here
  $\theta \colon (I \cdot 1) \times (I \cdot 1) \rightarrow (I \times
  I) \cdot 1$ is the canonical isomorphism characterised by
  $\theta \cdot (\iota_i, \iota_j) = \iota_{(i,j)}$.

  So the hyperaffine operations of $\mathbb{T}_X$ are those of
  the form
  $(h, \mathrm{id}) \colon X \rightarrow (1 \cdot I) \times X$; and,
  of course, the unary operations are those of the form
  $m \colon X \rightarrow X$. Moreover, if $h$ is hyperaffine and $m$
  is unary, then the operation $\dc h m = h(\lambda i.\, m(i))$ is
  interpreted by the composite
  \begin{equation*}
    X \xrightarrow{(h, \mathrm{id})} (1 \cdot I) \times X \xrightarrow{\mathrm{id} \times m} (1 \cdot I) \times X \quad = \quad X \xrightarrow{(h,m)} (1 \cdot I) \times X\rlap{ .}
  \end{equation*}
  So each operation $t = (h,m) \in T(I)$ has a unique
  hyperaffine--unary decomposition into the hyperaffine
  $(h,\mathrm{id}) \in T(I)$ and the unary $m \in T(1)$, as desired.
\end{proof}

\section{Matched pairs of theories}
\label{sec:match-pairs-theor}

As the name suggests, a hyperaffine--unary theory has a hyperaffine
part and a unary part. In this section, we show that these two parts,
together with their actions on each other, provide an entirely
equivalent description of the notion of hyperaffine--unary theory. To
begin with, let us record how we extract out the two parts of a
hyperaffine--unary theory.

\begin{Prop}
  \label{prop:14}
  Let $\mathbb{T}$ be a hyperaffine--unary theory. The hyperaffine
  operations of $\mathbb{T}$ form a subtheory $\mathbb{H}$, called the
  \emph{hyperaffine part}; while $T(1)$ forms a monoid $M$ under
  substitution, called the \emph{unary part}.
\end{Prop}

\begin{proof}
  The only non-trivial point is that hyperaffine operations are closed
  under substitution. But every affine operation is hyperaffine by
  Lemma~\ref{lem:2}, and the affine operations in any theory are
  easily closed under substitution.
\end{proof}

Clearly, if we know $\mathbb{H}$ and $M$ then we know every operation
of $\mathbb{T}$. However, the monoid structure of $M$ and the
substitution structure of $\mathbb{H}$ do \emph{not} determine the
substitution structure of $\mathbb{T}$. For this, we also need to
record how $\mathbb{H}$ and $M$ act on each other via substitution in
$\mathbb{T}$.

\begin{Prop}
  \label{prop:15}
  Let $\mathbb{T}$ be a hyperaffine--unary theory with hyperaffine and
  unary parts $\mathbb{H}$ and $M$. We may determine operations
  \begin{equation*}
    \begin{aligned}
      M \times H(I) & \rightarrow H(I) & \qquad \quad H(I) \times M^I & \rightarrow M \\
      (m,h) & \mapsto m^\ast h & (h,n) & \mapsto h \rhd n
    \end{aligned}
  \end{equation*}
  as follows: if $h \in H(I)$ and
  $m \in M$, then $m^\ast h \in H(I)$ is unique such that
  \begin{equation}\label{eq:12}
    m(h(\lambda i.\, x_i)) = (m^\ast h)(\lambda i.\, m(x_i))\rlap{ ;}
  \end{equation}
  while if
  $h \in H(I)$ and $n \in M^I$, then $h \rhd n \in M$ is unique
  such that
  \begin{equation}\label{eq:13}
    h(\lambda i.\, n_i(x)) = (h \rhd n)(x)\rlap{ .}
  \end{equation}
  These operations uniquely determine the substitution
  structure of $\mathbb{T}$ via the formulae:
  \begin{itemize}
  \item $\pi_i = \dc {\pi_i} 1$ in $T(I)$, where $1 = \pi_1$
    is the identity element of $M$;
  \item If $\dc h m \in T(I)$ and $\dc k n \in T(J)^I$ (i.e.,
    $\dc {k_i} {n_i} \in T(J)$ for each $i$), then 
    \begin{equation}\label{eq:14}
      \dc h m(\dc k n) = \dc {h(\lambda i.\, m^\ast k_i)} {h \rhd (\lambda i.\, mn_i)} \text{ in } T(J)\rlap{ .}
    \end{equation}
  \end{itemize}

\end{Prop}
\begin{proof}
  For the unique existence of $m^\ast h$, the composite operation
  $m(h(\lambda i.\, x_i)) \in T(I)$ admits a unique hyperaffine--unary
  decomposition; but since $m(h(\lambda i.\, x)) = m(x)$, the unary
  part of this must be $m$. The corresponding hyperaffine part
  $m^\ast h \in H(I)$ is thus the unique operation making~\eqref{eq:12}
  hold. The unique existence of $h \rhd n$ is trivial: take it as the
  substitution $h(n)$ in $\mathbb{T}$, and then~\eqref{eq:13} follows
  from associativity of substitution.
  It remains to prove that these operations determine the projections
  and substitutions for $\mathbb{T}$. That $\pi_i = \dc {\pi_i} 1 $ is
  trivial; as for~\eqref{eq:14}, we calculate that:
  \setlength{\belowdisplayskip}{-12pt}
  \begin{align*}
    \dc h m(\dc k n)(x) &= h(\lambda i.\, m(k_i(\lambda j.\, n_i(x_j)))) &
    \text{definition of $\dc \ \ $} & \quad \\
    &= h(\lambda i.\, (m^\ast k_i)(\lambda j.\, m(n_i(x_j)))) & \text{definition of $m^\ast k_i$}\\
    &= h(\lambda i.\, h(\lambda i'.\, (m^\ast k_i)(\lambda j.\, m(n_{i'}(x_j))))) & \text{hyperaffineness of $h$}\\
    &= h(\lambda i.\, (m^\ast k_i)(\lambda j.\, h(\lambda i'.\, m(n_{i'}(x_j))))) & \text{commutativity in $\mathbb{H}$}\\
    &= (h(\lambda i.\, m^\ast k_i))(\lambda j.\, h(\lambda i.\, mn_{i})(x_j)) &  \text{associativity in $\mathbb{H}$}\\
    &= \dc {h(\lambda i.\, m^\ast k_i)} {h \rhd (\lambda i.\, mn_i)}(x)  & \text{definition of $\dc \ \ $}\rlap{.}
  \end{align*}
\end{proof}

So any hyperaffine--unary algebraic theory is determined by its
hyperaffine and unary parts, together with the operations
of~\eqref{eq:12} and~\eqref{eq:13}. However, if given a hyperaffine
theory $\mathbb{H}$ and a monoid $M$, together with operations of the
same form, we should \emph{not} expect to obtain a structure of
hyperaffine--unary theory on the sets $T(I) = H(I) \times M$: for
although the preceding proposition indicates how to define
substitution from these operations, it does not ensure that the axioms
of a theory are satisfied. For this, we must impose axioms on the
operations relating $\mathbb{H}$ and $M$.

\begin{Defn}[Matched pairs of theories]
  \label{def:14}
   A \emph{matched pair of theories}
  $\dc {\mathbb{H}} M$ comprises a
  hyperaffine theory $\mathbb{H}$ and a monoid $M$ together with
  operations:
  \begin{equation}\label{eq:15}
    \begin{aligned}
      M \times H(I) & \rightarrow H(I) & \qquad \quad H(I) \times M^I & \rightarrow M \\
      (m,h) & \mapsto m^\ast h & (h,n) & \mapsto h \rhd n
    \end{aligned}
  \end{equation}
  satisfying the following axioms:
  \begin{enumerate}[(i)]
  \item For $m \in M$, the maps
    $m^\ast(\thg) \colon H(I) \rightarrow H(I)$ give a homomorphism of
    algebraic theories
    ${m^\ast \colon \mathbb{H} \rightarrow \mathbb{H}}$:
    \begin{align}
        m^\ast(\pi_i) &= \pi_i \label{eq:16}\\
  m^\ast(h(k)) &= (m^\ast h)(\lambda i.\, m^\ast k_i)\rlap{ ;} \label{eq:17}
    \end{align}
  \item The maps in (i) constitute an $M$-action on
    $\mathbb{H}$:
    \begin{align}
        1^\ast h &= h \label{eq:18}\\
  m^\ast n^\ast h &= (mn)^\ast(h)\rlap{ ;} \label{eq:19}
    \end{align}
  \item The operations $\mathord{\rhd} \colon H(I) \times M^I \rightarrow
    M$ make $M$ into a $\mathbb{H}$-model $\mdl[M]$:
    \begin{align}
        \pi_i \rhd m &= m_i \label{eq:20} \\
  h(k) \rhd (m) &= h \rhd (\lambda i.\, k_i \rhd m)\rlap{ ;} \label{eq:21}
    \end{align}
  \item Right multiplication by $n \in M$ is a
    $\mathbb{H}$-model homomorphism $(\thg)n \colon \mdl[M]
    \rightarrow \mdl[M]$:
    \begin{equation}
  (h \rhd m) n = h \rhd (\lambda i.\, m_in)\rlap{ ;} \label{eq:22}
    \end{equation}
  \item Left multiplication by $n \in M$ is a $\mathbb{H}$-model
    homomorphism ${n(\thg) \colon \mdl[M] \rightarrow n^\ast\mdl[M]}$,
    where $n^\ast \mdl[M]$ is the $\mathbb{H}$-model obtained by
    pulling back $\mdl[M]$ along the theory homomorphism $n^\ast
    \colon \mathbb{H} \rightarrow \mathbb{H}$:
    \begin{equation}
  n(h \rhd m) = n^\ast h \rhd (\lambda i.\, nm_i)\rlap{ ;} \label{eq:23}
    \end{equation}
  \item For $h \in H(I)$, the map
    $(\thg)^\ast h \colon M \rightarrow H(I)$ is a $\mathbb{H}$-model
    homomorphism $\mdl[M] \rightarrow \mdl[H](I)$, where $\mdl[H](I)$ is the
  free $\mathbb{H}$-model structure on $H(I)$:
    \begin{equation}
      (k \rhd m)^\ast(h) = k(\lambda j.\, m_j^\ast h)\rlap{ .}\label{eq:24}
    \end{equation}
  \end{enumerate}
  We call a matched pair of theories \emph{finitary} or
  \emph{non-degenerate} when $\mathbb{H}$ is so.

  A \emph{homomorphism}
  $\dc \varphi f \colon \dc {\mathbb{H}} M \rightarrow \dc
  {\mathbb{H}'} {M'}$ of matched pairs of theories is a homomorphism
  of theories $\varphi \colon \mathbb{H} \rightarrow \mathbb{H}'$ and
  a monoid homomorphism $f \colon M \rightarrow M'$ such that for all
  $h \in H(I)$, $m \in M$ and $n \in M^I$, we have:
  \begin{equation}\label{eq:25}
    \varphi(m^\ast h) = f(m)^\ast(\varphi(h)) \quad \text{and} \quad f(h \rhd n) = \varphi(h) \rhd (\lambda i.\, f(n_i))\rlap{ .}
  \end{equation}
\end{Defn}

We now show soundness and completeness of this axiomatisation.

\begin{Prop}
  \label{prop:16}
  \looseness=-1
  If $\mathbb{T}$ is a hyperaffine--unary theory, then its hyperaffine
  and unary parts $\mathbb{H}$ and $M$ constitute a matched pair of
  theories $\mathbb{T}^\downarrow = \dc {\mathbb{H}} M$ under the
  operations of~\cref{prop:15}; and $\mathbb{T}^\downarrow$ is
  finitary or non-degenerate just when $\mathbb{T}$ is so.
\end{Prop}
\begin{proof}
    For~\eqref{eq:16}--\eqref{eq:19}, we calculate using~\eqref{eq:12}
  and the theory axioms of $\mathbb{T}$ and conclude using unicity of
  hyperaffine--unary decompositions. For~\eqref{eq:16}, the
  calculation is
  $(m^\ast \pi_i)(\lambda j.\, m(x_j)) = m(\pi_i(\lambda j.\, x_j)) =
  m(x_i) = \pi_i(\lambda j.\, m(x_j))$. 
  For~\eqref{eq:17}:
  \begin{align*}
    (m^\ast(t(u)))(\lambda j.\, m(x_j)) &= 
    m(t(\lambda i.\, u_i(x))) = (m^\ast t)(\lambda i.\, m(u_i(x))) \\
    & =
    (m^\ast t)(\lambda i.\, (m^\ast u_i)(\lambda j.\, m(x_j))))
    \\ &=
    ((m^\ast t)(\lambda i.\, m^\ast u_i))(\lambda j.\, m(x_j))\rlap{ .}
  \end{align*}
  For~\eqref{eq:18}, we have
  $(1^\ast h)(x) = (1^\ast h)(\lambda i.\, 1(x_i)) = 1(h(x)) = h(x)$,
  and for~\eqref{eq:19},
  \begin{align*}
    (m^\ast n^\ast h)(\lambda i.\, m(n(x_i))) &= 
    m((n^\ast h)(\lambda i.\, n(x_i)) = 
    m(n(h(x))) \\ &= (mn)(h(x)) = ((mn)^\ast h)(\lambda i.\, m(n(x_i)))\rlap{ .}
  \end{align*}
  Next,~\eqref{eq:20} and~\eqref{eq:21} follow directly from the
  theory axioms for $\mathbb{T}$. Finally,
  for~\eqref{eq:22}--\eqref{eq:24}, we calculate
  using~\eqref{eq:12} and \eqref{eq:13} and conclude using unicity of
  decompositions. For~\eqref{eq:22} the calculation is that
  $(h \rhd m)(n)(x) = h(\lambda i.\, m_i(n(x))) = (h \rhd (\lambda
  i.\, m_i n))(x)$. For~\eqref{eq:23} we have:
  \begin{equation*}
    (n(h \rhd m))(x) = n(h(\lambda i.\, m_i(x)) = (n^\ast h)(\lambda i.\, n(m_i(x))) = (n^\ast h \rhd (\lambda i.\, nm_i))(x)\rlap{ ;}
  \end{equation*}
  and finally,
  for~\eqref{eq:24} we have:
  \begin{align*}
    ((k \rhd m)^\ast h)(\lambda i.\, (k \rhd m)(x_i)) &= 
    (k \rhd m)(h(x)) = k(\lambda j.\, m_j(h(x))) \\
    &= k(\lambda j.\, (m_j^\ast h)(\lambda i.\, m_j(x_i))) \\
    &= k(\lambda j.\, k(\lambda j'.\, (m_j^\ast h)(\lambda i.\, m_{j'}(x_i)))) \\
    &= k(\lambda j.\, (m_j^\ast h)(\lambda i.\, k(\lambda j'.\, m_{j'}(x_i)))) \\
    &= (k(\lambda j.\, (m_j^\ast h)))(\lambda i.\, (k \rhd m)(x_i))\rlap{ .}\qedhere
  \end{align*}
\end{proof}
\begin{Prop}
  \label{prop:17}
  For any matched pair of theories $\dc {\mathbb{H}} M$, there is
  a hyperaffine--unary theory $\mathbb{H} \bc M$, the
  \emph{bicrossed product} of $\mathbb{H}$ and $M$, with
  $(\mathbb H \bc M)^\downarrow \cong \dc {\mathbb{H}} M$.
\end{Prop}
\begin{proof}
  For each set $I$, we take $(\mathbb{H} \bc M)(I) = H(I) \times M$, and write a
  typical element like before as $\dc h m$. We define projection
  elements $\pi_i = \dc {\pi_i} 1$ and substitution operations by the
  formula~\eqref{eq:14}, and claim that, upon doing so, we obtain an
  algebraic theory $\mathbb{H} \bc M$. For this, we must check the three
  theory axioms. Firstly:
    \begin{equation*}
      \dc h m(\lambda i.\, \dc {\pi_i} 1) = \dc {h(\lambda i.\, m^\ast
        \pi_i)} {h \rhd (\lambda i.\, m1)} = \dc {h(\lambda i.\, \pi_i)}
      {h \rhd (\lambda i.\, m)} = \dc h m
  \end{equation*}
  by the definition, axiom~\eqref{eq:16}, and axiom~\eqref{eq:21}.
  Secondly,
  \begin{equation*}
    \dc {\pi_i} 1 (\dc k n) = \dc {\pi_i(\lambda j.\, 1^\ast k_j)} {\pi_i \rhd (\lambda j.\, 1n_j)} = \dc {1^\ast k_i} {1n_i} = \dc {k_i} {n_i}\rlap{ .}
  \end{equation*}
  by the definition, axiom~\eqref{eq:20} and axiom~\eqref{eq:18}.
  Finally, for associativity of substitution, we first compute that
  $(\dc h m(\dc k n))(\dc \ell p)$ is given by
  \begin{align*}
     &{} \mathrel{\phantom{=}}
     \dc {h(\lambda i.\, m^\ast k_i)} {h \rhd (\lambda i.\, mn_i)}(\dc \ell p) \\ &=
     \dc {(h(\lambda i.\, m^\ast k_i))(\lambda j.\, (h \rhd (\lambda i.\, mn_i))^\ast \ell_j)} {h(\lambda i.\, m^\ast k_i) \rhd (\lambda j.\, (h \rhd (\lambda i.\, mn_i))p_j)}
  \end{align*}
  while $\dc h m(\lambda i.\, \dc {k_i} {n_i}(\dc \ell p))$ is given by
  \begin{align*}
    &{} \mathrel{\phantom{=}} \dc h m(\lambda i.\, \dc {k_i(\lambda j.\, n_i^\ast \ell_j)} {k_i \rhd (\lambda j.\, n_ip_j)})\\
    &= \dc {h(\lambda i.\, m^\ast(k_i(\lambda j.\, n_i^\ast \ell_j)))} {h \rhd (\lambda i.\, m(k_i \rhd (\lambda j.\, n_ip_j)))}\rlap{ .}
  \end{align*}
  Comparing first terms we have:
  \begin{align*}
    &{} \mathrel{\phantom{=}} (h(\lambda i.\, m^\ast k_i))(\lambda j.\, h \rhd (\lambda i.\, mn_i)^\ast \ell_j) \\
    &= h(\lambda i.\, m^\ast k_i(\lambda j.\, (h \rhd (\lambda i'.\, mn_{i'}))^\ast \ell_j)) & \text{associativity in $\mathbb{H}$}\\
    &= h(\lambda i.\, m^\ast k_i(\lambda j.\, h(\lambda i'.\, (mn_{i'})^\ast \ell_j))) & \text{\eqref{eq:24}}\\
    &= h(\lambda i.\, h(\lambda i'.\, m^\ast k_i(\lambda j.\, (mn_{i'})^\ast \ell_j))) & \text{commutativity in $\mathbb{H}$}\\
    &= h(\lambda i.\, m^\ast k_i(\lambda j.\, (mn_{i})^\ast \ell_j)) &
    \text{hyperaffinness}\\
    &= h(\lambda i.\, m^\ast k_i(\lambda j.\, m^\ast n_{i}^\ast \ell_j)) &
    \text{\eqref{eq:19}}\\
    &= h(\lambda i.\, m^\ast(k_i(\lambda j.\, n_{i}^\ast \ell_j)))\rlap{ ;} &
    \text{\eqref{eq:17}}
  \end{align*}
  while comparing second terms, we have
  \begin{align*}
    &{} \mathrel{\phantom{=}} h(\lambda i.\, m^\ast k_i) \rhd (\lambda j.\, (h \rhd (\lambda i.\, mn_i))p_j) \\
    &= h(\lambda i.\, m^\ast k_i) \rhd (\lambda j.\, h \rhd (\lambda i.\, mn_ip_j)) & \text{\eqref{eq:22}}\\
    &= h \rhd (\lambda i.\, m^\ast k_i \rhd (\lambda j.\, h \rhd (\lambda i'.\, mn_{i'}p_j))) &
    \text{\eqref{eq:21}} \\
    &= h \rhd (\lambda i.\, h \rhd (\lambda i'.\, m^\ast k_i \rhd (\lambda j.\, mn_{i'}p_j))) &
    \text{commutativity in $\mathbb{H}$} \\
    &= h \rhd (\lambda i.\, m^\ast k_i \rhd (\lambda j.\, mn_{i}p_j))) &
    \text{hyperaffineness} \\
    &=h \rhd (\lambda i.\, m(k_i \rhd (\lambda j.\, n_ip_j))) &
    \text{\eqref{eq:23}}
  \end{align*}
  as desired. So $\mathbb{H} \bc M$ is an algebraic theory.

  We next characterise the unary and hyperaffine operations of
  $\mathbb{H} \bc M$. Clearly the unary operations are those of the
  form ${\dc 1 m}$. As for the hyperaffines, note that
  $\dc h m \in (\mathbb{H} \bc M)(I)$ will be hyperaffine when
  ${\dc h m (\lambda i.\, x) = x}$, i.e., when
  $\dc {h(\lambda i.\, x)} m = \dc {x} {1}$. In particular, we must
  have $m = 1$; and such an element will be hyperaffine just when also
  $\dc h 1(\lambda i.\, \dc h 1(\lambda j.\, x_{ij})) = \dc h
  1(\lambda i.\, x_{ii})$, i.e.,
  $\dc {h(\lambda i.\, h(\lambda j.\, x_{ij}))} 1 = \dc {h(\lambda
    i.\, x_{ii})} 1$. Since each $h \in H(I)$ is hyperaffine, we
  conclude that the hyperaffines in $(\mathbb{H} \bc M)(I)$ are all
  elements of the form $\dc h 1$. It follows from this
  characterisation that each $\dc h m \in (\mathbb{H} \bc M)(I)$ has
  the unique hyperaffine--unary decomposition
  $\dc h m(x) = \dc h 1(\lambda i.\, \dc 1 m(x_i))$, whence it follows
  that $\mathbb{H} \bc M$ is a hyperaffine--unary theory.

  It remains to show that
  $(\mathbb{H} \bc M)^\downarrow \cong \dc {\mathbb{H}} M$. Writing
  ${\mathbb{H}'}$ and ${M'}$ for the hyperaffine and unary parts of
  $\mathbb{H} \bc M$, we have isomorphisms
  $\varphi_I \colon H(I) \rightarrow H'(I)$ and
  $f \colon M \rightarrow M'$ given by $h \mapsto \dc h 1$ and
  $m \mapsto \dc 1 m$. Because
  $\dc h 1(\lambda i.\, \dc {k_i} 1) = \dc {h(k)} 1$ and
  $\pi_i = \dc {\pi_i} 1$ in $\mathbb{H} \bc M$, the maps $\varphi_I$
  constitute an isomorphism of theories
  $\mathbb{H} \rightarrow \mathbb{H}'$; and because
  $\dc 1 m(\dc 1 n) = \dc 1 {mn}$ and $1 = \pi_1 = \dc 1 1$ in $T(1)$,
  the map $f$ is a monoid isomorphism $M \rightarrow M'$.

  We now verify the two axioms in~\eqref{eq:25}. For the first,
  observe that the operation $(\thg)^\ast$ on
  $(\mathbb{H} \bc M)^\downarrow$ has $(\dc 1 m)^\ast(\dc h 1)$ given
  by the unique element $\dc k 1 \in H'(I)$ for which
  $\dc k 1(\lambda i.\, \dc 1 m(x_i)) = \dc 1 m(\dc h 1)$. But
  $\dc k 1(\lambda i.\, \dc 1 m(x_i)) = \dc k m$ and
  $\dc 1 m(\dc h 1) = \dc {m^\ast h} 1$, whence
  $(\dc 1 m)^\ast(\dc h 1) = \dc {m^\ast h} 1$, i.e.,
  $f(m)^\ast(\varphi(h)) = \varphi(m^\ast h)$ as required. For the
  second axiom in~\eqref{eq:25}, note that the operation $\rhd$ on
  $(\mathbb{H} \bc M)^\downarrow$ is given by
  $\dc h 1 \rhd (\lambda i.\, \dc 1 {m_i}) = \dc h 1(\lambda i.\, \dc
  1 {m_i}) = \dc 1 {h \rhd m}$, which says that
  $\varphi(h) \rhd (\lambda i.\, f(m_i)) = f(h \rhd m)$ as required.

  Finally, it is trivial to observe that $\mathbb{T}$ is finitary or
  non-degenerate if and only if $\mathbb{H}$ is so, i.e., if and only
  if $\mathbb{T}^\downarrow$ is so.
\end{proof}

So we have a correspondence between hyperaffine--unary theories and
matched pairs of theories; we now describe how this correspondence
interacts with semantics.

\begin{Defn}[Models of matched pairs of theories]
  \label{def:15}
  Let $\dc {\mathbb{H}} M$ be a matched pair of theories. A $\dc
  {\mathbb{H}} M$-model $\mdl[X]$ is a set $X$ endowed with
  both $\mathbb{H}$-model structure $h, x \mapsto \dbr{h}(x)$ and $M$-set
  structure $m, x \rightarrow m \cdot x$ in such a way that
  \begin{equation}\label{eq:26}
    (h \rhd m) \cdot x = \dbr{h}(\lambda i.\, m_i \cdot x) \qquad \text{and} \qquad
    n \cdot \dbr{h}(x) = \dbr{n^\ast h}(\lambda i.\, n \cdot x)
  \end{equation}
  for all $h \in H(I)$, $x \in X^I$, $m \in M^I$ and $n \in M$; while
  a homomorphism of $\dc {\mathbb{H}} M$-models is a function
  preserving both $\mathbb{H}$-model and $M$-set structure. We write
  $\dc {\mathbb{H}} M\text-\cat{Mod}$ for the variety of
  $\dc {\mathbb{H}} M$-models.
\end{Defn}

\begin{Prop}
  \label{prop:18}
  Let $\mathbb{T}$ be a hyperaffine--unary theory with
  $\mathbb{T}^\downarrow = \dc {\mathbb{H}} M$. The variety of
  $\mathbb{T}$-models is concretely isomorphic to the variety of $\dc
  {\mathbb{H}} M$-models.
\end{Prop}
\begin{proof}
  Firstly, restricting back a $\mathbb{T}$-model structure on
  a set $X$ to the hyperaffine and unary parts yields
  $\mathbb{H}$-model and $M$-set structure which satisfy the axioms
  in~\eqref{eq:26} due to the definitions of the operations
  $(\thg)^\ast$ and $\rhd$ in $\mathbb{T}^\downarrow$.
  Conversely, $\mathbb{H}$-model and $M$-set structure on $X$ yields
  $\mathbb{T}$-model structure with
  \begin{equation}\label{eq:27}
    \dbr{ \,\dc h m \,}(x) = \dbr{h}(\lambda i.\, m \cdot x_i)\rlap{ .}
  \end{equation}
  The projection axioms hold as every projection is in
  $\mathbb{H}$. As for substitution:
    \begin{align*}
    &{}\phantom{=} {\ } \dbr{\, \dc h m\, }(\lambda i.\, \dbr{\, \dc {k_i} {n_i}}(x))\\
    &= \dbr{h}(\lambda i.\, m \cdot \dbr{k_i}(\lambda j.\, n_i \cdot x_j)) & \text{definition}\\
    &= \dbr{h}(\lambda i.\, \dbr{m^\ast k_i}(\lambda j.\, mn_i \cdot x_j)) & \text{\eqref{eq:26}}\\
    &= \dbr{h}(\lambda i.\, \dbr{h}(\lambda i'.\, \dbr{m^\ast k_i}(\lambda j.\, mn_{i'} \cdot x_j))) & \text{hyperaffinness}\\
    &= \dbr{h}(\lambda i.\, \dbr{m^\ast k_i}(\lambda j.\, \dbr{h}(\lambda i'.\, mn_{i'} \cdot x_j))) & \text{commutativity}\\
    &= \dbr{h}(\lambda i.\, \dbr{m^\ast k_i}(\lambda j.\, (h \rhd (\lambda i'.\, mn_{i'})) \cdot x_j)) & \text{\eqref{eq:26}}\\
    &= \dbr{h(\lambda i.\, m^\ast k_i)}(\lambda j.\, (h \rhd (\lambda i.\, mn_{i})) \cdot x_j)) & \text{$\mathbb{H}$-model axiom}\\
    &= \dbr{\, \dc {h(\lambda i.\, m^\ast k_i)} {h \rhd (\lambda i.\, mn_i)}\,}(x) & \text{definition}\\
    &= \dbr{\, \dc h m(\dc k n) }(x) & \text{\eqref{eq:14}.}
  \end{align*}
  It is easy to see that these assignments are mutually inverse;
  and in light of~\eqref{eq:27}, the homomorphisms match up under the correspondence.
\end{proof}

To conclude this section, we make the correspondence between
hyperaffine--unary theories and matched pairs of theories functorial.
We write:
\begin{itemize}[itemsep=0.25\baselineskip]
\item $\dc {\HA} {\Un}$ (resp., $\dc {\HA^\omega} {\Un}$) for
  the category of non-degenerate matched pairs (resp., finitary
  matched pairs) of theories and their homomorphisms;
\item $\HAUn$ (resp., $\HAUnf$) for the full subcategory of $\AT$ on
  the non-degenerate hyperaffine--unary (resp.\, finitary
  hyperaffine--unary) theories;
\item $\CCVar$ (resp., $\CCVar^\omega$) for the full subcategory of
  $\Var$ on the non-degenerate cartesian closed (resp., cartesian
  closed finitary) varieties.
\end{itemize}

Now the assignment
$\dc {\mathbb{H}} M \mapsto \dc {\mathbb{H}} M\text-\cat{Mod}$ can be
made functorial. Indeed, given a homomorphism
$\dc \varphi f \colon \dc {\mathbb{H}} M \rightarrow \dc {\mathbb{H}'}
{M'}$ of non-degenerate matched pairs of theories, we have a concrete
functor
${\dc \varphi f}^\ast \colon \dc {\mathbb{H}'} {M'}\text-\cat{Mod}
\rightarrow \dc {\mathbb{H}} {M}\text-\cat{Mod}$ which acts by
$\varphi^\ast$ and $f^\ast$ on the $\mathbb{H}'$-model and $M'$-set
structures. In this way, we obtain a functor
$(\thg)\text-\cat{Mod} \colon \dc{\HA}{\Un}^\mathrm{op}
\rightarrow \Var$ which, in light of Proposition~\ref{prop:18}, the
final clause of Proposition~\ref{prop:16}, and Theorem~\ref{thm:3},
must land inside $\CCVar$.

The assignment $(\thg)^\downarrow$ of
Proposition~\ref{prop:16} can also be made functorial:

\begin{Prop}
  \label{prop:19}
  The assignment $\mathbb{T} \mapsto \mathbb{T}^\downarrow$ is the
  action on objects of a functor
  $(\thg)^\downarrow \colon \HAUn \rightarrow \dc {\HA} {\Un}$, which
  on morphisms takes
  $\varphi \colon \mathbb{S} \rightarrow \mathbb{T}$ to the homomorphism
  $\dc {\res \varphi {\mathbb{H}}} {\res \varphi M} \colon
  \mathbb{S}^\downarrow \rightarrow \mathbb{T}^\downarrow$.
\end{Prop}
\begin{proof}
  $(\thg)^\downarrow$ is clearly functorial so long as it is
  well-defined on morphisms. To show this, let
  $\varphi \colon \mathbb{T} \rightarrow \mathbb{T}'$ be a
  homomorphism between hyperaffine--unary theories. Clearly, $\varphi$
  preserves both hyperaffine operations and unary operations, and so
  restricts back to
  $\res \varphi {\mathbb{H}} \colon \mathbb{H} \rightarrow
  \mathbb{H}'$ and $\res \varphi M \colon M \rightarrow M'$. We must
  verify that these restrictions satisfy the axioms in~\eqref{eq:25}.
  The second axiom is simply an instance of the homomorphism axiom for
  $\varphi$; as for the first, we have:
  \begin{align*}
    \varphi(m^\ast h)\bigl(\lambda i.\, \varphi(m)(x_i)\bigr) &= 
    \varphi\bigl( m^\ast h(\lambda i.\, m(x_i))\bigr) = 
    \varphi\bigl( m(h(x))\bigr) \\
    &= \varphi(m)\bigl(\varphi(h)(x)\bigr) = \bigl(\varphi(m)^\ast\varphi(h)\bigr)\bigl(\lambda i.\, \varphi(m)(x_i)\bigr)
  \end{align*}
  whence
  $\res \varphi {\mathbb{H}}(m^\ast h) =
  (\res \varphi M(m))^\ast(\res \varphi {\mathbb{H}}(h))$ by unicity of
  decompositions.
\end{proof}

Given the above, we are now ready to state the main result of this section:
\begin{Thm}
  \label{thm:4}
  We have a triangle of equivalences, commuting to
  within natural isomorphism, as to the left in:
  \begin{equation*}
    \cd[@!C@C-3em]{
      \HAUn \ar[dr]_-{(\thg)\text-\mathrm{Mod}} \ar[rr]^-{(\thg)^\downarrow} & & 
      \dc {\HA}{\Un} \ar[dl]^-{(\thg)\text-\mathrm{Mod}} \\
      & (\CCVar)^\mathrm{op}
    } \qquad 
    \cd[@!C@C-3em]{
      \HAUnf \ar[dr]_-{(\thg)\text-\mathrm{Mod}} \ar[rr]^-{(\thg)^\downarrow} & & 
      \dc {\HA^\omega}{\Un} \ar[dl]^-{(\thg)\text-\mathrm{Mod}} \\
      & (\CCVar^\omega)^\mathrm{op}
    }
  \end{equation*}
  which restricts back to a triangle of equivalences as to the right.
\end{Thm}

\begin{proof}
  The triangles commute to within isomorphism by
  Proposition~\ref{prop:18}, and by Theorem~\ref{thm:3}, their left
  edges are equivalences. So to complete the proof it suffices to show
  that $(\thg)^\downarrow$ is an equivalence to the left and the right. We
  know that in both cases it is essentially surjective by
  Proposition~\ref{prop:17}, and so it remains only to show it is also
  full and faithful. For fidelity, note that any homomorphism
  $\varphi \colon \mathbb{S} \rightarrow \mathbb{T} \in \HAUn$ must,
  by unicity of decompositions, send $\dc h m$ to
  $\dc {\varphi(h)} {\varphi(m)}$, and so is determined by its
  hyperaffine and unary restrictions. For fullness, let
  $\mathbb{S}, \mathbb{T} \in \HAUn$ and let
  $\dc \varphi f \colon \mathbb{S}^\downarrow \rightarrow
  \mathbb{T}^\downarrow$. We must show that the functions
  \begin{equation*}
    \psi \colon S(I) \rightarrow T(I) \qquad \dc h m \mapsto \dc {\varphi(h)}{f(m)}
  \end{equation*}
  preserve projections and substitution. For projections, we have
  $\psi(\pi_i) = \varphi(\dc {\pi_i} 1) =
  \dc{\varphi(\pi_i)}{f(1)} = \dc
  {\pi_i} 1 = \pi_i$, while for substitution, we have:
  \begin{align*}
    \psi(\dc h m(\dc k n)) &= 
    \psi(\dc {h(\lambda i.\, m^\ast k_i)} {h \rhd (\lambda i.\, mn_i)}) \\
    &= 
    \dc {\varphi(h(\lambda i.\, m^\ast k_i))} {f(h \rhd (\lambda i.\, mn_i))} \\
    &= 
    \dc {(\varphi(h))(\lambda i.\, \varphi(m^\ast k_i)))} {\varphi(h) \rhd (\lambda i.\, f(mn_i))}\\
    &= 
    \dc {(\varphi(h))(\lambda i.\, f(m)^\ast(\varphi(k_i)))} {\varphi(h) \rhd (\lambda i.\, f(m)f(n_i))} \\
    &= \dc {\varphi(h)} {f(m)}(\lambda i.\,
    \dc {\varphi(k)} {f(n)}) \\
    &= \psi(\dc h m)(\lambda i.\, \psi(\dc {k_i} {n_i}))\rlap{ .}\qedhere
  \end{align*}
\end{proof}

\section{Matched algebras and $\BM$-sets}
\label{sec:general-case:-two}

In this section we use Theorem~\ref{thm:1} to recast the notion of
matched pair of theories in terms of what we will call a \emph{matched
  pair of algebras}. This yields a reformulation of
Theorem~\ref{thm:4} giving a functorial equivalence between
(non-degenerate) hyperaffine--unary theories, matched pairs of
algebras, and cartesian closed varieties. We begin in the finitary
case, where a matched pair $\BM$ involves a Boolean algebra $B$ and a
monoid $M$ which act on each in a suitable way---a structure which was
already considered in~\cite[\sec 4]{Jackson2009Semigroups}, in a
related, though different, context.

\begin{Defn}[Matched pair of algebras]
  \label{def:16}
  A non-degenerate \emph{matched pair of algebras} $\BM$ comprises
  a non-degenerate Boolean algebra $B$, a monoid $M$ and:
  \begin{itemize}
  \item $B$-set structure on $M$, which we write as $b, m, n \mapsto b(m,n)$;
  \item $M$-set structure on $B$, which we write as $m, b \mapsto m^\ast b$;
  \end{itemize}
  such that $M$ acts on $B$ by Boolean homomorphisms, and such that:
  \begin{itemize}
  \item $b(m,n)p = b(mp,np)$;
  \item $m(b(n,p)) = (m^\ast b)(mn,mp)$; and
  \item $b(m,n)^\ast(c) = b(m^\ast c, n^\ast c)$,
  \end{itemize}
  for all $m,n,p \in M$ and $b,c \in B$. Here, in the final axiom, we
  recall from Remark~\ref{rk:1} that $B$ itself is a $B$-set under the
  operation of conditioned disjunction
  $b(c,d) = (b \wedge c) \vee (b' \wedge d)$. These axioms are
  equivalently the conditions that:
  \begin{itemize}
  \item $m \equiv_b n  \implies mp \equiv_b np$;
  \item $n \equiv_b p  \implies mn \equiv_{m^\ast b} mp$;
  \item $m \equiv_b n \implies m^\ast c \equiv_b n^\ast c$, i.e., $b
    \wedge m^\ast c = b \wedge n^\ast c$.
  \end{itemize}

  A \emph{homomorphism of matched pairs of algebras}
  $\dc \varphi f \colon \BM \rightarrow \dc {B'} {M'}$ comprises a
  Boolean homomorphism $\varphi \colon B \rightarrow B'$ and a monoid
  homomorphism $f \colon M \rightarrow M'$ such that, for all
  $m,n \in M$ and $b \in B$ we have:
  \begin{equation}\label{eq:28}
    \varphi(b)(f(m),f(n)) = f(b(m,n)) \ \ \  \text{and} \ \ \ f(m)^\ast(\varphi(b)) = \varphi(m^\ast b)\rlap{ ,}
  \end{equation}
  or equivalently, such that
  \begin{equation}\label{eq:29}
    \ \ \quad m \equiv_b n \implies f(m) \equiv_{\varphi(b)} f(n) \qquad \text{and} \qquad f(m)^\ast(\varphi(b)) = \varphi(m^\ast b)\text{ .} 
  \end{equation}
  We write $\dc {\BA} {\Mon}$ for the category of non-degenerate
  matched pairs of algebras.
\end{Defn}

We now establish the desired equivalence between finitary matched
pairs of theories, and matched pairs of algebras.

\begin{Prop}
  \label{prop:20}
  The assignment sending a non-degenerate finitary hyperaffine theory
  $\mathbb{H}$ to the Boolean algebra $B = H(2)$ of
  Proposition~\ref{prop:10} induces an equivalence of categories
  $\Theta \colon \dc {\HA^\omega} {\Un} \rightarrow \dc {\BA} {\Mon}$.
\end{Prop}
\begin{proof}
  The assignment $\mathbb{H} \mapsto H(2)$ is, by Theorem~\ref{thm:1},
  the action on objects of an equivalence
  $\HA^\omega \rightarrow \BA$. Under this equivalence,
  the data and axioms of a matched pair of theories
  $\dc {\mathbb{H}} M$ as in Definition~\ref{def:14} transform as follows:
  \begin{itemize}
  \item $\mathbb{H}$ and $M$ correspond to the Boolean algebra $B =
    H(2)$ and monoid $M$;
  \item The maps to the left of~\eqref{eq:15}, satisfying the
    bicrossed pair axioms
    (i) and (ii), correspond to a monoid action $m, b \mapsto m^\ast b$ of
    $M$ on $B$ by Boolean homomorphisms;
  \item The maps to the right of~\eqref{eq:15}, satisfying the axiom
    (iii), correspond to a $B$-set structure $b,m,n \mapsto b(m,n)$ on
    $M$;
  \item The axioms (iv) and (v) correspond directly to the first two
    axioms for a matched pair of algebras.
  \end{itemize}
  As for axiom (vi), we claim that this corresponds to the final axiom
  for a matched pair of algebras. This is not completely immediate: we
  must first observe that (vi) can be replaced by the apparently
  weaker special case which takes $I = 2$:
  \begin{enumerate}
  \item[(vi)$'$] For $h \in H(2)$, the map
    $(\thg)^\ast h \colon M \rightarrow H(2)$ is an $\mathbb{H}$-model map
    $\mdl[M] \rightarrow \mdl[H](2)$.
  \end{enumerate}
  To see that this special case implies the general one, we must show
  that for any $h \in H(I)$, $k
  \in H(J)$ and $m \in M^J$ we have:
  $(k \rhd m)^\ast(h) = k(\lambda j.\, m_j^\ast h)$.
  By~\cref{prop:11}(i), it suffices to verify for each $i \in I$ the
  equality of the
  binary reducts $(\thg)^{(i)}$ of each side. Since the operations
  $(k \rhd m)^\ast$ and $m_j^\ast$ are theory homomorphisms
  $\mathbb{H} \rightarrow \mathbb{H}$, we have
  \begin{equation*}
    ((k \rhd m)^\ast h)^{(i)} = (k \rhd m)^\ast(h^{(i)}) \quad \text{and} \quad k(\lambda j.\, m_j^\ast h)^{(i)} = k(\lambda j.\, m_j^\ast(h^{(i)}))
  \end{equation*}
  and since the $h^{(i)}$ are binary, these terms are equal by the
  special case (vi)$'$. Observing that the $\mathbb{H}$-model
  structure on $\mathbf{H}(2)$ corresponds to the $B$-action on $B$ by
  conditioned disjunction, we thus conclude that (vi)$'$, and hence
  also (vi), are equivalent to the final axiom for a matched pair of theories.

  It remains to show that homomorphisms match up under the above
  correspondences: for which we must show that the two conditions
  of~\eqref{eq:25} correspond to the two conditions of~\eqref{eq:28}.
  For the first condition in~\eqref{eq:25}, this is achieved by
  exploiting~\cref{prop:11}(i) like before to reduce to then
  case $I = 2$. As for the second condition, we may
  re-express it as saying that $f$ is a homomorphism
  of $\mathbb{H}$-models
  $\mdl[{M}] \rightarrow \varphi^\ast(\mdl[{M'}])$, from which the
  correspondence with~\eqref{eq:28} is immediate.
\end{proof}

So each finitary matched pair of theories $\dc {\mathbb{H}} M$ has a more
concrete expression as a matched pair of algebras $\BM$; we now
show that, correspondingly, the
variety of $\dc {\mathbb{H}} M$-models has a more concrete
expression as a variety of \emph{$\BM$-sets}:

\begin{Defn}[Variety of $\BM$-sets]
  \label{def:17}
  Let $\BM$ be a non-degenerate matched pair of algebras. A
  \emph{$\BM$-set} is a set $X$ endowed with $B$-set structure and
  $M$-set structure, such that in addition we have:
  \begin{equation}\label{eq:30}
    b(m,n) \cdot x = b(m \cdot x, n \cdot x) \qquad \text{and} \qquad m \cdot b(x,y) = (m^\ast b)(m \cdot x, m \cdot y)
  \end{equation}
  for all $b\in B$, $m,n \in M$ and $x,y \in X$; or equivalently, such
  that:
  \begin{equation}
    \label{eq:53}
    m \equiv_b n \implies m \cdot x \equiv_b n \cdot x \qquad \text{and} \qquad
    x \equiv_b y \implies m \cdot x \equiv_{m^\ast b} m \cdot y\rlap{ .}
  \end{equation}
  A homomorphism of
  $\BM$-sets is a function which is at once a $B$-set and an $M$-set
  homomorphism. We write $\BM\text-\cat{Set}$ for the variety of
  $\BM$-sets.
\end{Defn}

We noted in the preceding sections that any non-degenerate Boolean
algebra $B$ is always a $B$-set over itself, and that any monoid $M$
is always an $M$-set over itself. If $\BM$ is a matched pair, then
by definition we also have that $B$ is an $M$-set, and $M$ is a
$B$-set; it should therefore be no surprise that, when endowed with
these actions, both $B$ and $M$ become $\BM$-sets. In fact, $M$ is
the free $\BM$-set on one generator; while $B$ is the coproduct of
two copies of the terminal $\BM$-set.

\begin{Prop}
  \label{prop:21}
  The equivalence $\Theta$ of Proposition~\ref{prop:20} fits into a triangle of
  equivalences, commuting to within natural isomorphism:
  \begin{equation*}
    \cd[@!C@C-3em]{
      \dc {\HA^\omega}{\Un} \ar[dr]_-{(\thg)\text-\mathrm{Mod}} \ar[rr]^{\Theta} & & 
      \dc{\BA} {\Mon} \ar[dl]^-{(\thg)\text-\mathrm{Set}} \\
      & (\CCVar^\omega)^\mathrm{op}\rlap{ .}
    }
  \end{equation*}
\end{Prop}
\begin{proof}
  Given a matched pair of theories
  $\dc {\mathbb{H}} M \in \dc {\HA^\omega}{\Un}$ with associated
  matched pair of algebras $\BM$, we know by Theorem~\ref{thm:1}
  that the data of an $\dc{\mathbb{H}} M$-model structure on $X$, as
  in Definition~\ref{def:15}, will
  transform as follows:
  \begin{itemize}
  \item The $\mathbb{H}$-model and $M$-set structure on $X$ correspond
    to a $B$-set structure $b,x,y \mapsto b(x,y)$ and an $M$-set
    structure;
  \item The left-hand axiom in~\eqref{eq:26}, after reducing to the
    case $I = 2$ as in the proof of Proposition~\ref{prop:20}, becomes
    the left-hand axiom in~\eqref{eq:30}.
  \item The right-hand axiom in~\eqref{eq:26} states that $n \cdot
    (\thg)$ is an $\mathbb{H}$-model homomorphism $\mdl[X] \rightarrow
    n^\ast \mdl[X]$, and thus becomes the right-hand
    axiom in~\eqref{eq:30}.
  \end{itemize}
  It is clear that the homomorphisms match up under this
  correspondence, and so the variety of $\dc {\mathbb{H}} M$-models
  and the variety of $\BM$-sets are concretely isomorphic. It is
  easy to check that these isomorphisms are natural in
  $\dc {\mathbb{H}} M$ as required.
\end{proof}

Now extending the preceding arguments to the non-finitary case is
straightforward. First we generalise the notion of matched pair of
algebras.

\begin{Defn}
  \label{def:18}
  A non-degenerate \emph{Grothendieck matched pair of algebras} $\BJM$
  comprises a non-degenerate matched pair of algebras $\BM$ and a
  zero-dimensional topology $\J$ on $B$, such that:
  \begin{itemize}
  \item The $B$-set $M$ is a $\BJ$-set;
  \item The $M$-action on $B$ is by Grothendieck Boolean
    homomorphisms $\BJ \rightarrow \BJ$.
  \end{itemize}
  
  A \emph{homomorphism of Grothendieck matched pairs of algebras}
  $\dc \varphi f \colon \BJM \rightarrow \dc {B'_{\J'}} {M'}$
  is a homomorphism of matched pairs of algebras for which
  $\varphi \colon \BJ \rightarrow B'_{\J'}$ is a Grothendieck Boolean
  homomorphism. We write $\dc {\GBA} {\cat{Mon}}$ for the
  category of non-degenerate Grothendieck matched pairs of algebras.
\end{Defn}
We next generalise the correspondence between finitary hyperaffine
theories and matched pairs of algebras; the proof of this result is
\emph{mutatis mutandis} the same as Proposition~\ref{prop:20}.

\begin{Prop}
  \label{prop:22}
  The assignment sending a non-degenerate hyperaffine theory
  $\mathbb{H}$ to the Grothendieck Boolean algebra $\BJ$ of
  Proposition~\ref{prop:12} induces an equivalence of categories
  $\Theta \colon \dc {\HA} {\Un} \rightarrow \dc {\GBA} {\Mon}$. \qed
\end{Prop}

Finally, we introduce the varieties associated to Grothendieck matched
pairs of algebras $\BJM$, and show that they match up with the
models of the corresponding matched pairs of theories.

\begin{Defn}[Variety of $\BJM$-sets]
  \label{def:19}
  Let $\BJM$ be a non-degenerate Grothendieck matched pair of
  algebras. A \emph{$\BJM$-set} is a $\BM$-set $X$ whose
  underlying $B$-set is in fact a $\BJ$-set; a homomorphism of
  $\BJM$-sets is just a homomorphism of 
  $\BM$-sets. We write $\BJM\text-\cat{Set}$ for the
  variety of $\BJM$-sets.
\end{Defn}

Like before, both $B$ and $M$ are $\BJM$-sets via their
canonical actions on themselves and each other.

\begin{Prop}
\label{prop:23}
  The equivalence ${\Theta}$ of Proposition~\ref{prop:22} fits into a triangle of
  equivalences, commuting to within natural isomorphism:
  \begin{equation*}
    \cd[@!C@C-3em]{
      \dc {\HA}{\Un} \ar[dr]_-{(\thg)\text-\mathrm{Mod}} \ar[rr]^{\Theta} & & 
      \dc{\GBA} {\Mon} \ar[dl]^-{(\thg)\text-\mathrm{Set}} \\
      & (\CCVar)^\mathrm{op}\rlap{ .}
    } \rlap{\qquad \qquad \ \ \quad \text{\qed}}
  \end{equation*}
\end{Prop}

\begin{Rk}
  \label{rk:2}
  We can extract from the above development a description of the free
  $\BJM$-set on a set $X$ as given by the product of
  $\BJM$-sets $M \times T_{\BJ} X$. Here, $M$ is seen as a
  $\BJM$-set via its canonical structures of $\BJ$- and
  $M$-set, while $T_{\BJ}(X)$ is seen as a $\BJ$-set as in
  Remark~\ref{rk:1} and as an $M$-set via the action
  $n \cdot (m, \omega) = (nm, n^\ast \circ \omega)$. The function
  $\eta \colon X \rightarrow M \times T_{\BJ}(X)$ exhibiting
  $M \times T_{\BJ}(X)$ as free on $X$ is given by
  $x \mapsto (1, \pi_x)$.
\end{Rk}

Combining Propositions~\ref{prop:22} and \ref{prop:23} with
Theorem~\ref{thm:4}, we obtain the main theorem of this section,
relating (non-degenerate) Grothendieck matched pairs, 
hyperaffine--unary theories and cartesian closed
varieties.

\begin{Thm}
  \label{thm:5}
  We have a triangle of equivalences, commuting to within natural
  isomorphism, as to the left in:
  \begin{equation*}
    \cd[@!C@C-4.8em]{
      \HAUn  \ar[dr]_-{(\thg)\text-\cat{Mod}} \ar[rr]^-{(\thg)^\downarrow} & & 
       \dc {\GBA}{\Mon}\ar[dl]^-{(\thg)\text-\cat{Set}} \\
      & (\CCVar)^\mathrm{op}
    } \qquad
    \cd[@!C@C-4.5em]{
      \HAUnf \ar[dr]_-{(\thg)\text-\cat{Mod}} \ar[rr]^-{(\thg)^\downarrow} & & 
      \dc {\BA}{\Mon} \ar[dl]^-{(\thg)\text-\cat{Set}}  \\
      & (\CCVar^\omega)^\mathrm{op}
    }
  \end{equation*}
  which restricts back to a triangle of equivalences as to the right.
\end{Thm}

Another way to say this is that every non-degenerate cartesian closed
variety is a variety of $\BJM$-sets for some Grothendieck matched pair
of algebras $\BJM$. We now make explicit the cartesian closed
structure of $\BJM\text-\cat{Set}$.

\begin{Prop}
  \label{prop:24}
  The variety of $\BJM$-sets is cartesian closed.
  In particular, the variety of $\BM$-sets is cartesian closed.
\end{Prop}
\begin{proof}
  Given $\BJM$-sets $Y$ and $Z$, we define the function space $Z^Y$ to
  be the set of $\BJM$-set homomorphisms
  $f \colon M \times Y \rightarrow Z$. We make this into an
  $M$-set under the same action as in Proposition~\ref{prop:1}:
  \begin{equation*}
    m, f \qquad \mapsto \qquad m^\ast f = (\lambda n, y.\, f(nm, y))\rlap{ .}
  \end{equation*}
  We must check $m^\ast f$ is a $\BJM$-set homomorphism if $f$ is one.
  For the $M$-set aspect this is just as in Proposition~\ref{prop:1};
  for the $\BJ$-set aspect, if $n \equiv_b p$ and $y \equiv_b z$ then
  $nm \equiv_b pm$ and so $f(nm, y) \equiv_b f(pm,y)$, i.e.,
  $(m^\ast f)(n,y) \equiv_b (m^\ast f)(p,z)$ as desired.
  We now make $Z^Y$ into a $\BJ$-set via the equivalence relations:
  \begin{equation*}
    f \equiv_b g \qquad \iff \qquad f(m,y) \equiv_{m^\ast b} g(m,y) \text{ for all $m,y \in M \times Y$.}
  \end{equation*}
  Axiom (i) of Proposition~\ref{prop:2} is straightforward, given
  that each $m^\ast$ is a Grothendieck Boolean homomorphism; so it
  now suffices to check axiom (ii)$'$ of Proposition~\ref{prop:6}.
  Thus, given a partition $P \in \J$ and homomorphisms
  $f_b \colon M \times Y \rightarrow Z$ for each $b \in P$, we must
  show there is a unique $g \colon M \times Y \rightarrow Z$ with
  $g \equiv_{b} f_b$ for all $b \in P$, i.e.,
  \begin{equation}\label{eq:31}
    g(m,y) \equiv_{m^\ast b} f_b(m,y) \qquad \text{for all $b \in B$.}
  \end{equation}
  As $m^\ast$ is a Grothendieck Boolean homomorphism, the set
  $m^\ast P = \{m^\ast b : b \in B\}^-$ is in $\J$, and so for each
  $(m,y)$ there is a \emph{unique} element $g(m,y)$
  satisfying~\eqref{eq:31}. It remains to show that the
  $g \colon M \times Y \rightarrow Z$ so defined is a $\BJM$-set
  homomorphism.

  To see $g$ preserves the $\BJ$-set structure, suppose that
  $m \equiv_c n$ in $M$ and $y \equiv_c z$ in $Y$; we must show
  $g(m,y) \equiv_c g(n,z)$ in $Z$. Since $m \equiv_c n$ we have for
  each $b \in P$ that $c \wedge m^\ast b = c \wedge n^\ast b$, and so
  $g(m,y) \equiv_{c \wedge m^\ast b} f_b(m,y) \equiv_{c \wedge m^\ast
    b} f_b(n,z) \equiv_{c \wedge m^\ast b} g(n,z)$, using that
  $g(m,y) \equiv_{m^\ast b} f_b(m,y)$ and $f_b(m,y) \equiv_c f_b(n,z)$
  and $f_b(n,z) \equiv_{n^\ast b} g(n,z)$. Thus, for all
  $m^\ast b \in m^\ast P$ we have
  $g(m,y) \equiv_{c \wedge m^\ast b} g(n,z)$ and so by
  Lemma~\ref{lem:1}(i) that $g(m,y) \equiv_c g(n,z)$ as required.
  To see $g$ preserves the $M$-set structure, we must show
  $m \cdot g(n,y) = g(mn,my)$. But for each $b \in P$ we have
  $g(n,y) \equiv_{n^\ast b} f_b(n,y)$, and so
  $m \cdot g(n,y) \equiv_{(mn)^\ast b} m \cdot f_b(n,y) = f_b(mn,my)
  \equiv_{(mn)^\ast b} g(mn,my)$. Thus $m \cdot g(n,y) = g(mn,my)$ by
  Lemma~\ref{lem:1}(i).

  So $Z^Y$ is a well-defined $\BJM$-set. We now define the evaluation
  homomorphism $\mathrm{ev} \colon Z^Y \times Y \rightarrow Z$ as in
  Proposition~\ref{prop:1} by $\mathrm{ev}(f, y) = f(1, y)$. This
  preserves $M$-set structure as there; while for the $\BJ$-set
  structure, if $f \equiv_b g$ in $Z^Y$ and $y \equiv_b y'$ in $Y$,
  then
  $\mathrm{ev}(f,y) = f(1, y) \equiv_b g(1,y) \equiv_b g(1,y') =
  \mathrm{ev}(g,y')$ as desired.

  Finally, given a $\BJM$-set homomorphism
  $f \colon X \times Y \rightarrow Z$, its transpose
  $\bar f \colon X \rightarrow Z^Y$ is given by
  $\bar f(x)(m, y) = f(mx, y)$. As in Proposition~\ref{prop:1}, this
  is an $M$-set homomorphism, and is the unique such with
  $\mathrm{ev}(\bar f(x), y) = f(x,y)$ for all $x,y$.
  It remains to show that $\bar f$ preserves $\BJ$-set structure. But if
  $x \equiv_b x'$, then $mx \equiv_{m^\ast b} mx'$ for all $m \in M$,
  and so
  $\bar f(x)(m,y) = f(mx,y) \equiv_{m^\ast b} f(mx',y) = \bar
  f(x')(m,y)$ for all $(m,y) \in M \times Y$, i.e.,
  $\bar f(x) \equiv_b \bar f(x')$.
\end{proof}

Note that the results of the preceding sections have shown, without
reference to~\cite{Johnstone1990Collapsed}, that every non-degenerate
hyperaffine--unary theory is a theory of $\BJM$-sets. Since every
degenerate hyperaffine--unary theory clearly presents a cartesian
closed variety, the preceding result thus completes a proof of the
``if'' direction of Theorem~\ref{thm:3} that does not rely
on~\cite{Johnstone1990Collapsed}. Taken together with
Proposition~\ref{prop:47}, we thus obtain our desired independent
proof of Theorem~\ref{thm:3}.

As mentioned in the introduction, we defer substantive examples of
varieties of $\BJM$-sets to the companion
paper~\cite{Garner2023CartesianII}, where we establish links with
topics in operator algebra. However, Proposition~\ref{prop:47} assures
us that there is a plentiful supply of such varieties: we have one for
any object $X$ of a category with finite products and distributive
set-indexed copowers. We can now be more explicit about the $\BJM$
associated to such an $X$.

\begin{Prop}
  \label{prop:44}
  Let $\C$ be a category with finite products and set-based copowers
  for which each functor $C \times (\thg)$ preserves copowers, and let
  $X \in \C$.
  \begin{enumerate}[(a),itemsep=0.25\baselineskip]
  \item   The set $\M = C(X,X)$ is a monoid with unit
    $\mathrm{id}_X$ under the
    operation of composition in \emph{diagrammatic} order, i.e.:
    \begin{equation*}
      mn \quad = \quad X \xrightarrow{m} X \xrightarrow{n} X\rlap{ ;}
    \end{equation*}
  \item Writing $\iota_\top, \iota_\bot \colon 1 \rightarrow 1+1$ for
    the first and second copower coprojections, the set
    $B = \C(X,1 + 1)$ is a Boolean algebra under the operations
    \begin{gather*}
      \smash{1 = X \xrightarrow{!} 1 \xrightarrow{\iota_\top} 1+1 \qquad \qquad b' = X \xrightarrow{b} 1 + 1 \xrightarrow{\spn{\iota_2, \iota_1}} 1+1}\\
      \text{and} \quad 
      \smash{b \wedge c = X \xrightarrow{(b,c)} (1 + 1) \times (1 + 1) \xrightarrow{\wedge} 1+1}
    \end{gather*}
    where $\wedge \colon (1+1) \times (1+1) \rightarrow 1+1$ satisfies
    $\wedge \circ (\iota_i \times \iota_j) = \iota_{i \wedge j}$ for
    $i,j \in \{\top, \bot\}$;
  \item There is a zero-dimensional coverage $\J$ on $B$ in which
    $P \subseteq B$ is in $\J$ just when there exists a map
    $f \colon X \rightarrow P \cdot 1$ with
    $\spn{\delta_{bc}}_{b \in B} \circ f = c$ for all $c \in P$, where
    here $\delta_{bc} \colon 1 \rightarrow 1+1$ is given by
    $\delta_{bc} = \iota_{\top}$ when $b = c$ and
    $\delta_{bc} = \iota_\bot$ otherwise;
  \item $M$ acts on $B$ via precomposition;
    \begin{equation*}
      \smash{m^\ast b \quad = \quad X \xrightarrow{m} X \xrightarrow{b} 1+1\rlap{ ;}}
    \end{equation*}
  \item $B$ acts on $M$ via:
    \begin{equation*}
      \smash{(b, m, n) \mapsto X \xrightarrow{(b, \mathrm{id})} (1+1) \times X \xrightarrow{\cong} X + X \xrightarrow{\spn{m,n}} X\rlap{ .}}
    \end{equation*}
  \end{enumerate}
  So long as $B$ is non-degenerate, the above operations make $\BJM$
  into a non-degenerate Grothendieck matched pair of algebras.
  Moreover, for all $Y \in \C$, the set $\C(X,Y)$ becomes a
  $\BJM$-set, where $M$ acts on $\C(X,Y)$ via precomposition, and $B$
  acts on $\C(X,Y)$ via
  \begin{equation*}
    \smash{(b, x, y) \mapsto X \xrightarrow{(b, \mathrm{id})} (1+1) \times X \xrightarrow{\cong} X + X \xrightarrow{\spn{x,y}} Y\rlap{ .}}
  \end{equation*}
  In this manner, we obtain a factorisation of the
  hom-functor $\C(X, \thg)$ as
  \begin{equation}
    \label{eq:32}
    \cd[@!C@C-4em@-0.5em]{
      {\C} \ar@{-->}[rr]^-{K} \ar[dr]_-{\C(X, \thg)} & &
      {\BJM\text-\cat{Set}} \ar[dl]^-{U} \\ &
      {\cat{Set}}
    }
  \end{equation}
  which is universal among factorisations of $\C(X, \thg)$ through a
  variety. In particular, if $\C$ is a non-degenerate cartesian closed
  variety and $X$ is the free model $F1$ on one generator, then $K$ is
  an isomorphism.
\end{Prop}
\begin{proof}
  By Proposition~\ref{prop:47}, the complete theory $\mathbb{T}_X$ of
  dual operations of $X \in \C$ is hyperaffine--unary. If
  $\mathbb{T}_X$ is degenerate, then so is the Boolean algebra $B$
  described above, and there is nothing to do; otherwise, we know that
  the non-degenerate hyperaffine--unary $\mathbb{T}_X$ corresponds to
  a Grothendieck matched pair $\BJM$ for which
  $\mathbb{T}\text-\cat{Mod}$ is concretely isomorphic to
  $\BJM\text-\cat{Set}$. Using that the hyperaffine operations of
  $\mathbb{T}_X$ are, as in the proof of Proposition~\ref{prop:47},
  those of the form
  $(h,1) \colon X \rightarrow (I \cdot 1) \times X \cong I
  \cdot X$, and following through the construction of $\BJM$ from the
  hyperaffine--unary theory $\mathbb{T}_X$ as in
  Sections~\ref{sec:match-pairs-theor}
  and~\ref{sec:general-case:-two}, yields the description above.
  Finally, the factorisation~\eqref{eq:32} is simply the
  factorisation~\eqref{eq:34} after transporting across the
  isomorphism~$\mathbb{T}\text-\cat{Mod}\cong\BJM\text-\cat{Set}$.
\end{proof}

We encourage the reader to apply this result in any category
satisfying its rather mild hypotheses. For example, when $\C$ is the
category of topological spaces, the monoid $M$ associated to a space
$X$ comprises all continuous endomorphisms of $X$, while the
Grothendieck Boolean algebra $B_\J$ comprises all clopen subsets of
$X$, with the infinite partitions in $\J$ being all infinite clopen
partitions of $X$. Now $M$ acts on $B$ by inverse image,
$\varphi, U \mapsto \varphi^{-1}(U)$, while $B$ acts on $M$ by
restriction and glueing:
$U, f, g \mapsto \spn{\res f U, \res g {U^c}}$. In this example,
$\BJM$ is rather large, in much the same way that full automorphism
groups of objects tend to be rather large, and a key aspect
of~\cite{Garner2023CartesianII} will be to apply this result in
carefully chosen situations where $\BJM$ comes out as something
combinatorially tractable and of independent interest.

\bibliographystyle{acm}
\bibliography{bibdata}

\end{document}